\documentclass{amsart}
\usepackage[latin1]{inputenc}
\usepackage[T1]{fontenc}
\usepackage{lmodern}
\usepackage[english]{babel}
\usepackage{microtype}

\usepackage{amsmath,amssymb,amsfonts,amsthm}
\usepackage{mathtools,accents}
\usepackage{mathrsfs}
\usepackage{aliascnt}
\usepackage{braket}
\usepackage{bm}
\usepackage{lineno}
\usepackage[citecolor=blue,colorlinks]{hyperref}

\usepackage{enumerate}
\usepackage{xcolor}

\usepackage{aliascnt}

\makeatletter
\def\newaliasedtheorem#1[#2]#3{
  \newaliascnt{#1@alt}{#2}
  \newtheorem{#1}[#1@alt]{#3}
  \expandafter\newcommand\csname #1@altname\endcsname{#3}
}
\makeatother

\theoremstyle{plain}
\newtheorem{theorem}{Theorem}[section]
\newaliasedtheorem{lemma}[theorem]{Lemma}
\newaliasedtheorem{prop}[theorem]{Proposition}
\newaliasedtheorem{claim}[theorem]{Claim}
\newaliasedtheorem{corollary}[theorem]{Corollary}

\theoremstyle{definition}
\newaliasedtheorem{definition}[theorem]{Definition}
\newaliasedtheorem{example}[theorem]{Example}

\theoremstyle{remark}
\newaliasedtheorem{remark}[theorem]{Remark}

\numberwithin{equation}{section}

\def\Im{\textrm{Im}}
\def\Re{\textrm{Re}} 
\def\11{{\rm 1~\hspace{-1.4ex}l} }
\def\R{\mathbb R}
\def\C{\mathbb C}
\def\Z{\mathbb Z}
\def\N{\mathbb N}
\def\E{\mathbb E}
\def\T{\mathbb T}

\title
[NLS with spatial white noise ]
{Two dimensional nonlinear Schr\"odinger equation with spatial white noise potential and fourth order nonlinearity}
\author[N. Tzvetkov \and N. Visciglia]
{N. Tzvetkov \and N. Visciglia}

\address{N.~Tzvetkov, CY Cergy-Paris Universit\'e,  Cergy-Pontoise, F-95000, UMR 8088 du CNRS}
\email{nikolay.tzvetkov@cyu.fr}

\address{N.~Visciglia, Dipartimento di Matematica, Universit\`a di Pisa, Largo Bruno Pontecorvo, 5, 56100 Pisa, Italy}
\email{nicola.visciglia@unipi.it}

\begin{document}

\maketitle

\begin{abstract}
We consider NLS on $\T^2$ with multiplicative spatial white noise and nonlinearity between  cubic and quartic. We prove global existence, uniqueness and convergence almost surely of solutions to a family of properly
regularized and renormalized approximating equations. In particular we extend a previous result by A. Debussche and H. Weber available in the cubic and sub-cubic setting.
\end{abstract}
\section{Introduction}
\subsection{Statement of the main result}
We are interested in the following family of NLS with multiplicative spatial white noise:
\begin{equation}\label{NLSwhite}
i\partial_t u  = \Delta u + \xi u +\lambda u|u|^p, \quad
u(0,x)=u_0(x), \quad (t,x)\in \R\times \T^2
\end{equation}
where $\xi(x,\omega)$ is space white noise, $2\leq p\leq 3$, $\lambda\leq 0$, and we identify $\T^2$ with $(-\pi,\pi)\times (-\pi, \pi)$.
We work for simplicity with a defocusing nonlinearity, but the results of this paper can be extended to the focusing case under a smallness assumption on the initial datum. 
Our main aim is an improvement on the range of the nonlinearity $p$, from the case $p=2$ achieved by A. Debussche and H. Weber in \cite{ DW}, to the larger range $2\leq p \leq 3$.
We basically follow the approach of \cite{DW}, the main novelty being the introduction of modified energies in the context of \eqref{NLSwhite}. 
These energies allow to cover a larger set of $p$ in \eqref{NLSwhite}. They also have the potential to be useful in the future for the study of the growth of the high Sobolev norms in the context of \eqref{NLSwhite}.  
\\

We assume that $\xi(x,\omega)$ is real valued and has a vanishing zero Fourier mode (or equivalently of mean zero with respect to $x$). This assumption is however not essential because one may remove the zero mode of $\xi$ from the equation by the transform $u\mapsto e^{it\hat{\xi}(0)}u$.  Therefore in the sequel, we will assume that $\xi(x,\omega)$ is given by the following random Fourier series: 
$$
\xi(x,\omega)=\sum_{n\in\Z^2,n\neq 0}\, g_{n}(\omega)\,e^{in\cdot x}\,,
$$
where  $x\in \T^2$ and $(g_n(\omega))$ are identically distributed  standard complex gaussians on the probability space 
$(\Omega, {\mathcal F},p)$. We suppose that  $(g_n(\omega))_{n\neq 0}$ are independent, modulo the relation $\overline{g_n}(\omega)=g_{-n}(\omega)$ (so that $\xi$ is  a.s a real valued distribution). 
\\

Since the white noise $\xi(x,\omega)$ is not a classical function, it is important to properly define what we mean by a solution of \eqref{NLSwhite}.
The nature of the initial data $u_0(x)$ is also of importance in this discussion but even for $u_0(x)\in C^\infty(\T^2)$ it is not clear what we mean by a solution of \eqref{NLSwhite}. 
Let us therefore suppose first that $u_0(x)\in C^\infty(\T^2)$.  Since it is well known how to solve \eqref{NLSwhite} with $\xi(x,\omega)\in C^\infty(\T^2)$ and $u_0(x)\in C^\infty(\T^2)$, it is  natural to consider the following regularized problems: 
\begin{equation}\label{NLSwhite_bis}
i\partial_t u_\varepsilon  = \Delta u_\varepsilon + \xi_\varepsilon u_\varepsilon +\lambda u_\varepsilon |u_\varepsilon|^p\,\, , \quad u_{\varepsilon}(0,x)=u_0(x),
\end{equation}
where $\xi_\varepsilon(x,\omega)=\chi_{\varepsilon}(x)\ast \xi(x,\omega)$, $\varepsilon \in (0,1)$ is a regularization of $\xi$ by convolution with $\chi_{\varepsilon}(x)=\varepsilon^{-2}\chi(x/\varepsilon)$, where $\chi(x)$ is smooth with a support in $\{|x|<1/2\}$
and $\int_{\T^2}\chi dx=1$. Then we have
\begin{equation}\label{xiepsilon}
\xi_{\varepsilon}(x,\omega)=\sum_{n\in\Z^2,n\neq 0}\,  \rho\big(\varepsilon n\big) g_{n}(\omega)\,e^{in\cdot x}\,,
\end{equation}
where $\rho=\hat{\chi}$ is the Fourier transform on $\R^2$ of $\chi$.
\\

Unfortunately,  we do not know how to pass into the limit $\varepsilon\rightarrow 0$ in \eqref{NLSwhite_bis} (even for $u_0(x)\in C^\infty(\T^2)$) and it may be that this limit is quite singular in general.
Our analysis will show that we can only pass into the limit almost surely w.r.t. $\omega$ if we take a well chosen random approximation of the datum $u_0(x)$ in \eqref{NLSwhite_bis}
and if we properly renormalize the phase of the solution $u_\varepsilon(t,x,\omega)$. Following \cite{DW} and \cite{HL} we introduce the following smoothed potential
$Y=\Delta^{-1} \xi$ and its $C^\infty$ regularization 
$Y_\varepsilon=\Delta^{-1} \xi_\varepsilon$, namely:
\begin{equation}\label{Ysenzaepsilon}
Y(x,\omega)=-\sum_{n\in\Z^2,n\neq 0}\, \frac{g_{n}(\omega)}{|n|^2}\,e^{in\cdot x}
\end{equation}
and
\begin{equation}\label{Yepsilon}
Y_{\varepsilon}(x,\omega)=-\sum_{n\in\Z^2,n\neq 0}\,  \rho\big(\varepsilon n\big) \frac{g_{n}(\omega)}{|n|^2}\,e^{in\cdot x}.
\end{equation}
We now consider the regularized problems: 
\begin{equation}\label{NLSwhite_tris}
i\partial_t u_\varepsilon  = \Delta u_\varepsilon + \xi_\varepsilon u_\varepsilon +\lambda u_\varepsilon |u_\varepsilon|^p\,\, , \quad u_{\varepsilon}(0,x)=
u_0(x) e^{Y(x,\omega)-Y_\varepsilon(x,\omega)}
\end{equation}
where we assume that almost surely w.r.t. $\omega$ we have 
$e^{Y(x,\omega)} u_0(x)\in H^2(\T^2)$. Notice that under this assumption the problem
\eqref{NLSwhite_tris} has almost surely w.r.t. $\omega$ a classical unique global solution $u_\varepsilon(t,x,\omega)\in {\mathcal C}(\R; H^2(\T^2))$ 
(see \cite{Ts,BGT}).  Here is our main result. 
\begin{theorem} \label{corproof}   Assume $p\in [2,3], \lambda\leq 0$
and $u_0(x)$ be such that $e^{Y(x,\omega)} u_0(x) \in H^2(\T^2)$ a.s. 
Then there exists an event $\Sigma\subset \Omega$ such that $p(\Sigma)=1$ and
for every $\omega\in \Sigma$
there exists   
$$ 
v (t,x, \omega)\in \bigcap_{\gamma\in [0, 2)}{\mathcal C}(\R; H^\gamma(\T^2))
$$
such that for every $T>0$ and $\gamma \in [0,2)$ we have:
\begin{equation}\label{eas}
\sup_{t\in [-T,T]} \|e^{-iC_\varepsilon t} e^{Y_\varepsilon(x, \omega)} u_\varepsilon(t,x,\omega)- v (t,x,\omega)\|_{H^\gamma(\T^2)}
\overset{\varepsilon\rightarrow 0} \longrightarrow 0,
\end{equation}
where $C_\varepsilon=\sum_{n\in\Z^2,n\neq 0}\,
  \frac{\rho^2(\varepsilon n)}{|n|^2}$ and $u_\varepsilon(t,x,\omega)$
  are solutions to \eqref{NLSwhite_tris}.
Moreover for $\gamma\in [0,1)$ and $\omega\in \Sigma$ we have
\begin{equation}\label{diffnew}
\sup_{t\in [-T,T]} \big \||u_\varepsilon(t,x,\omega)| - e^{-Y(x,\omega)}| v(t,x, \omega)| \big \|_
{{H^\gamma(\T^2)}\cap L^\infty (\T^2)} \overset{\varepsilon\rightarrow 0}
\longrightarrow 0.
\end{equation}
\end{theorem}
The limits obtained in Theorem~\ref{corproof} are by definition what we may wish to call solutions of \eqref{NLSwhite} with datum $u_0(x)$.  
Observe that $|u_\varepsilon(t,x,\omega)|$ has a well defined limit, while the phase of $u_\varepsilon(t,x,\omega)$ should be suitably renormalized by the diverging constants $C_\varepsilon$ in order to get a limit. We also point out that the  meaning of the constants $C_\varepsilon$,  introduced along the statement of Theorem~\ref{corproof}, is explained in
Section~\ref{probability},  where the renormalization procedure is presented.
\\

It is worth mentioning that despite to \eqref{eas}, that works for $\gamma\in (0,2)$, in
\eqref{diffnew} we assume $\gamma\in (0,1)$. This is due to a technical reason
since, in order to estimate the Sobolev norm of the absolute value of a Sobolev function, 
we use the diamagnetic inequality which, to the best of our knowledge, works up to the $H^1$ regularity.
\\

In a future work \cite{TV_new}, we plan to extend the result of  Theorem~\ref{corproof} to any $p<\infty$ by exploiting the dispersive properties of the Schr\"odinger equation on a compact spatial domain established in \cite{BGT}. In fact, we shall not need to exploit the construction of \cite{BGT} in its full strength because we will only need an $\varepsilon$-improvement of the Sobolev embedding. This means that we will need to make the WKB construction of \cite{BGT}  for solutions oscillating at frequency $h^{-1}$ only up to time $h^{2-\delta}$, $\delta>0$ which are much shorter than the times $h$ achieved in \cite{BGT}. Such a room would hopefully allow us to incorporate in the dispersive estimates the first and zero order terms appearing after the application of the gauge transform (see the next section).
\\

In \cite{TV_new} we also plan to exploit the modified energy method, used indirectly in the proof of Theorem~\ref{corproof} 
(and more directly in the proof of Theorem \ref{main} below) in order to get polynomial bounds on higher Sobolev norms of the obtained solutions, similar to the ones obtained in \cite{PTV} in the case without a white noise potential. 
\\

In Theorem~\ref{corproof} the initial data $u_0(x)$ is well-prepared because it is supposed to satisfy $e^{Y(x,\omega)} u_0(x) \in H^2(\T^2)$ a.s.
It would be interesting to decide whether a suitable application of the $I$-method introduced in \cite{CKSTT} may allow to remove this assumption of well-prepared data. 
For this purpose, one should succeed to establish the limiting property by using energies at level  $H^s$, for a suitable $s<1$. 
 \subsection{The gauge transform}   
 In the sequel we perform some formal computations that allow us to introduce heuristically 
 a rather useful transformation.  Following \cite{DW} and \cite{HL} we introduce the new unknown:
 \begin{equation}\label{YY}
 v=e^{Y} u
 \end{equation}
 where $u$ is assumed to be formally solution to \eqref{NLSwhite} and $Y=\Delta^{-1} \xi$.
In order to clarify the relevance of this transformation first notice that by direct computation we have that the equation solved (at least formally) by $v$ is the following one: 
\begin{equation}\label{NLSwhitetransform}
i\partial_t v  = \Delta v  - 2\nabla v  \cdot \nabla Y+ v |\nabla Y|^2+\lambda e^{-pY}v|v|^p, \quad 
v(0,x)= e^{Y(x,\omega)} u_0(x).
\end{equation}
Notice that  the quantity $|\nabla Y|^2$ is not well defined
since $\nabla Y$, even if is one derivative more regular than $\xi$,
has still negative Sobolev regularity. However this issue can be settled 
by a renormalization  (see below and Section~\ref{probability} for more details). On the other hand \eqref{NLSwhitetransform} compared with \eqref{NLSwhite}
looks more complicated since a perturbation of order one in the linear part of the equation is added
compared with \eqref{NLSwhite}. Nevertheless, we have the advantage that the coefficients involved in the new equation are more regular that the spatial white noise $\xi$ that appears in \eqref{NLSwhite}.
\\

Another relevant advantage that comes from the new variable $v$ is related to the 
conservation of the Hamiltonian.
Recall that the conservation laws play a key role in the analysis of nonlinear Schr\"odinger equations.
In particular in the context of \eqref{NLSwhite} the quadratic part of the conserved energy is given by
\begin{equation}\label{ZZ}
\int_{\T^2} (|\nabla u |^2 - |u|^2 \xi) dx\,.
\end{equation}
The key feature in the transformation \eqref{YY} is that there is a cancellation between the two terms in \eqref{ZZ} and this cancellation
is the main point in the definition of a suitable self-adjoint realisation of $\Delta+\xi$ (see \cite{GUZ} and the references therein).  
Indeed, let us compute \eqref{ZZ} in the new variable $v$, hence we have
$u=e^{-Y} v$ and \eqref{ZZ} becomes 
$$
\int_{\T^2} (-e^{-Y} v \Delta (e^{-Y} \bar{v}) - e^{-2Y} |v|^2\xi)dx,
$$
which after some elementary manipulations can be written as:
$$
\int_{\T^2} ( |\nabla v|^2 +|v|^2 \Delta Y  - |v|^2 |\nabla Y|^2-  |v|^2\xi) e^{-2Y}dx\,.
$$
Thanks to the  choice $\Delta Y=\xi$ we get a cancellation 
of the white noise potential leading to 
$$
\int_{\T^2} ( |\nabla v|^2   - |v|^2 |\nabla Y|^2) e^{-2Y}dx\,.
$$
Notice that now the potential energy w.r.t. the new variable $v$ involves the potential 
$|\nabla Y(x)|^2$ which is (morally) one derivate more regular compared with the white noise.
\\

Motivated by the previous discussion, we observe that if $u_\varepsilon(t,x,\omega)$ is a solution to
\begin{equation*}\label{introwhitenoiseregnew}
i\partial_t u_\varepsilon  = \Delta u_\varepsilon + u_\varepsilon \xi_\varepsilon   +\lambda u_\varepsilon|u_\varepsilon|^p,
\end{equation*}
where $\xi_\varepsilon (x,\omega)$ is defined by \eqref{xiepsilon},
then the transformed function
\begin{equation}\label{gaugetranf}
v_\varepsilon(t,x,\omega)= e^{-iC_\varepsilon t} e^{Y_\varepsilon(x,\omega)} u_\varepsilon(t,x,\omega)
\end{equation}
satisfies
\begin{equation*}\label{NLSgaugeintro}
i\partial_t v_\varepsilon  = \Delta v_\varepsilon  - 2\nabla v_\varepsilon  \cdot \nabla Y_\varepsilon + v_\varepsilon :|\nabla Y_\varepsilon |^2:
+\lambda  e^{-pY_\varepsilon} v_\varepsilon|v_\varepsilon|^p.
\end{equation*}
Here we have $Y_\varepsilon (x,\omega)$ given by \eqref{Yepsilon} and $:|\nabla Y_\varepsilon |^2:(x,\omega)$ is defined as follows:
\begin{equation}\label{renormalization}
:|\nabla Y_\varepsilon |^2:(x,\omega)=|\nabla Y_\varepsilon |^2(x,\omega)-C_\varepsilon\end{equation}
where 
\begin{equation}\label{Cepsilon}
C_\varepsilon=\sum_{n\in\Z^2,n\neq 0}\,
 \frac{\rho^2(\varepsilon n)}{|n|^2}
 \end{equation} is the same constant as the one appearing in Theorem~\ref{corproof}.
One can show that almost surely w.r.t. $\omega$
we have the following convergence, in spaces with negative regularity:
$$:|\nabla Y_\varepsilon |^2:(x,\omega)\overset{\varepsilon\rightarrow 0} \longrightarrow :|\nabla Y|^2:(x,\omega),$$ 
where
\begin{multline}\label{:nabla:}:|\nabla Y|^2:(x,\omega)=\sum_{\substack{(n_1,n_2)\in\Z^4\\n_1\neq 0, n_2\neq 0\\n_1\neq n_2}}\,  
\frac{n_1\cdot n_2}{|n_1|^2 |n_2|^2}\,
g_{n_1}(\omega)\overline{g_{n_2}(\omega)}\,e^{i(n_1-n_2)\cdot x}
\\+\sum_{n\in\Z^2,n\neq 0}\,\frac{|g_n(\omega)|^2 -1}{|n|^2},\end{multline}
(see Section~\ref{probability} for details).\\

The main idea to establish Theorem \ref{corproof} is to look for the convergence of $v_\varepsilon$ as $\varepsilon\rightarrow 0$, and hence to get informations
on $u_\varepsilon$ by going back via the transformation \eqref{gaugetranf}.
\begin{theorem}\label{main} Assume $p\in [2,3], \lambda\leq 0$ and $u_0(x)$ be such that $e^{Y(x,\omega)} u_0(x) \in H^2(\T^2)$ a.s.
Then there exists an event $\Sigma\subset \Omega$ such that $p(\Sigma)=1$ and
for every $\omega\in \Sigma$  there exists $$ v(t,x,\omega)\in \bigcap_{\gamma\in [0, 2)}{\mathcal C}(\R; H^\gamma(\T^2))$$
such that  for every fixed $T>0$ and $\gamma\in [0,2)$ we have: 
$$\sup_{t\in [-T, T]} \|v_\varepsilon(t,x, \omega) - v(t,x, \omega)\|_{H^\gamma(\T^2)}
\overset{\varepsilon\rightarrow 0}\longrightarrow 0.$$
Here we have denoted by
$v_\varepsilon(t,x,\omega)$  for $\omega\in \Sigma$ the unique global solution in the space ${\mathcal C} (\R;H^2(\T^2))$ of the following problem:
\begin{multline}\label{NLSgaugeintronew}
i\partial_t v_\varepsilon  = \Delta v_\varepsilon  - 2\nabla v_\varepsilon  \cdot \nabla Y_\varepsilon + v_\varepsilon :|\nabla Y_\varepsilon |^2:
+\lambda  e^{-pY_\varepsilon}v_\varepsilon|v_\varepsilon|^p, \\
v_\varepsilon(0,x)=v_0(x)\in H^2(\T^2)
\end{multline}
and $ v(t,x, \omega)$ denotes for $\omega\in \Sigma$ the unique global solution in the space ${\mathcal C} (\R; H^\gamma(\T^2))$, for $\gamma\in (1, 2)$,
of the following limit problem:
\begin{multline}\label{NLSgaugeintronewbar}
i\partial_t  v  = \Delta v  - 2\nabla  v  \cdot \nabla Y +  v :|\nabla Y |^2:
+\lambda  e^{-pY} v| v|^p,\,\\
v(0,x)=v_0(x)\in H^2(\T^2)
\end{multline}
where in both Cauchy problems \eqref{NLSgaugeintronew}
and \eqref{NLSgaugeintronewbar} $v_0(x)=e^{Y(x,\omega)} u_0(x)$, $\omega\in \Sigma$.
\end{theorem}
The result of Theorem~\ref{main} for $p=2$, with a weaker convergence, was established in \cite{DW}. 
Here we still follow the strategy developed in \cite{DW} which can be summarized as follows:
\begin{enumerate}
\item A priori bounds for the $H^2$-norm of $v_\varepsilon$;
\item  Convergence of the special sequence $(v_{2^{-k}})$   a.s. w.r.t. $\omega$ in ${\mathcal C}([-T,T];H^\gamma(\T^2))$ for every $T>0$;
\item  Convergence of the whole family $(v_{\varepsilon})$  a.s. w.r.t. $\omega$ in ${\mathcal C}([-T,T];H^\gamma(\T^2))$ for every $T>0$;
\item  Pathwise uniqueness of solutions to  \eqref{NLSgaugeintronewbar}.
\end{enumerate}
In contrast with \cite{DW}, we do not use the pathwise uniqueness in the convergence procedure of steps $(2)$ and $(3)$.
The main novelty in this paper is that we can extend the $H^2$ bounds in step $(1)$  to the range of the nonlinearity $2\leq p\leq 3$. 
The key tool compared with \cite{DW} is the use of suitable energies in conjunction with the Br\' ezis-Gallou\"et inequality. 
This technique is inspired by \cite{Ts,OV,PTV}.  As already mentioned another difference compared with \cite{DW} is that we establish the convergence of solutions
to the regularized problems to the solution of the limit problem almost surely rather than in the weaker  convergence in probability.
\\

It would be interesting to decide whether the modified energy argument, developed in this paper can be useful in order to improve the range of the nonlinearity in \cite{DM},
where the NLS with multiplicative space white noise on the whole space is considered. 
\subsection{Notations}
Next we fix some notations.
We denote by $L^q$, $W^{s,q}$, $H^\gamma$,  the spaces
$L^q(\T^2)$, $W^{s,q}(\T^2)$, $H^\gamma(\T^2)$. 
Let us give the precise definition of $W^{s,q}$, we use. 
The linear operator $D^s$ is defined by
$$
D^s(e^{in\cdot x})=\langle n\rangle ^{s} e^{in\cdot x},
$$
where $\langle n\rangle=(1+|n|^2)^{\frac{1}{2}}$.  Then we define $W^{s,q}$ via the norm
$$
\|f\|_{W^{s,p}(\T^2)}:=\|D^s(f)\|_{L^p(\T^2)}\,.
$$
We also use the following notation for weighted Lebesgue spaces:
$\|f\|_{L^q(w)}^q=\int_{\T^2} |f|^q w \hbox{ } dx$ 
where $w\geq 0$ is a weight. We shall denote by $x=(x_1, x_2)$ the generic point in $\T^2$ and
$\nabla$ will be the full gradient operator
w.r.t. the space variables and also $\partial_i$ the partial derivative w.r.t. $x_i$. 
To simplify the presentation we denote by $\int_{\T^2} h$ the integral with respect
to the Lebesgue measure $\int_{\T^2} h \hbox{ } dx$. 
Starting from Section~3,  we will denote by $C(\omega)$ a generic random variable finite on the event of full probability defined in Proposition~\ref{ercol10}. The random contant $C(\omega)$ will be allowed to change from line to line
 in our computations. For every $q\in [1,\infty]$ we denote by $q'$ the conjugate H\"older exponent.  We shall use the notation $\lesssim$ in
order to denote a lesser or equal sign $\leq$ up to a positive multiplicative constant $C$, that in turn
may depend harmlessly on contextual parameters. In some cases we shall drop the dependence of the functions
from the variable $(t, x , \omega)$ when it is clear from the context.
\subsection{Plan of the remaining part of the paper}
In the next section, we present some stochastic analysis considerations.
Section~3 is devoted to the basic bounds resulting from the Hamiltonian structure and some variants of the Gronwall lemma. 
Section~4 contains the key bounds at $H^2$ level.  The proof of the algebraic proposition Proposition~\ref{modifen} is postponed to the last section. 
In Section~5, we present the proof of Theorem~\ref{main} while Section~6 is devoted to the proof of
Theorem~\ref{corproof}. In the final Section~7, we present the proof of Proposition~\ref{modifen}.
\section{Probabilistic results}\label{probability}
In this section we collect a series of results concerning the probabilistic object
$Y$ and its regularized version $Y_\varepsilon$ (see \eqref{Yepsilon}). The main point is that all the needed probabilistic properties
are established a.s., which is the key point to establish convergence a.s.
in Theorems~\ref{corproof} and \ref{main}.
We shall need in the rest of the paper some special random constants 
that will be a combination of the ones involved in Proposition \ref{specialrandom}.\\

First we justify the introduction of the constant $C_\varepsilon$ in \eqref{Cepsilon} as follows.
By definition of $Y_{\varepsilon}(x,\omega)$ (see \eqref{Yepsilon}) we have
$$
|\nabla Y_\varepsilon|^2(x,\omega)=\sum_{\substack{(n_1,n_2)\in\Z^4\\n_1\neq 0, n_2\neq 0}}\,  \rho\big(\varepsilon n_1\big)  \rho
(\varepsilon n_2) \frac{n_1\cdot n_2}{|n_1|^2 |n_2|^2}\,
g_{n_1}(\omega)\overline{g_{n_2}(\omega)}\,e^{i(n_1-n_2)\cdot x}
$$
whose zero Fourier coefficient is the random constant
$$\sum_{n\in\Z^2,n\neq 0}\,
  \rho^2(\varepsilon n)  \frac{|g_{n}(\omega)|^2}{|n|^2}.
$$
Hence the constant $C_\varepsilon$ defined in \eqref{Cepsilon}
is the average on $\Omega$ of the zero Fourier modes defined above.
We shall prove that a.s. w.r.t. $\omega$ the functions 
$:|\nabla Y_\varepsilon|^2:(x,\omega)$ defined in \eqref{renormalization}
converges as $\varepsilon\rightarrow 0$, in the topology $W^{-s,q}$ for $s\in (0,1)$ and $q\in (1, \infty)$, to the
limit object $:|\nabla Y|^2:(x,\omega)$ defined by \eqref{:nabla:}.
\\

Next we gather the key probabilistic properties that we need in the rest of the paper.
\begin{prop}\label{specialrandom} 
Let $s\in (0,1)$ and $q\in (1,\infty)$ be given. There exists an event $\Sigma_0\subset \Omega$ such that
$p(\Sigma_0)=1$ and for every $\omega\in \Sigma_0$ there exists a finite constant $C(\omega)>0$ such that:
\begin{itemize}
\item we have the following uniform bound:
\begin{multline*} \sup_{\varepsilon\in (0,1)} \big \{\|e^{\pm Y_\varepsilon}(x,\omega)\|_{L^\infty},
\|e^{\pm Y_\varepsilon}(x,\omega)\|_{W^{s,q}},\|\nabla Y_\varepsilon (x,\omega)\|_{L^q} |\ln \varepsilon|^{-1},\\
\|:|\nabla Y_\varepsilon |^2:(x,\omega) \|_{L^q} |\ln \varepsilon|^{-2},
\|:|\nabla Y_\varepsilon |^2: (x,\omega)\|_{W^{-s,q}}
\big \}<C(\omega);
\end{multline*}
\item for a suitable $\kappa>0$ we have:
\begin{equation}\label{edna_zvezda}
\|Y_\varepsilon (x,\omega)-Y(x,\omega)\|_{W^{s,q}}<C(\omega) \varepsilon^\kappa,
\end{equation}
in particular by choosing $sq>2$ we get by Sobolev embedding 
$$\|Y_\varepsilon (x,\omega)-Y(x,\omega)\|_{L^\infty}<C(\omega) \varepsilon^\kappa$$
and also
$$\|e^{-pY_\varepsilon (x,\omega)}-e^{-pY(x,\omega)}\|_{L^\infty}<C(\omega) \varepsilon^\kappa,
\quad \forall p\in \R;$$
\item for a suitable $\kappa>0$ we have
\begin{equation}\label{dve_zvezdi}
\|\nabla Y_\varepsilon (x,\omega) - \nabla Y(x,\omega) \|_{W^{-s,q}}< C(\omega) \varepsilon^\kappa,
\end{equation}
and
\begin{equation}\label{tri_zvezdi}
\|:|\nabla Y_\varepsilon|^2:(x,\omega) - :|\nabla Y|^2:(x,\omega) \|_{W^{-s,q}}< C(\omega) \varepsilon^\kappa.
\end{equation}
\end{itemize}
\end{prop}
Let us observe that since the condition on $s$ is open, by using the Sobolev embedding we can include the case $q=\infty$ in \eqref{edna_zvezda}, \eqref{dve_zvezdi}, \eqref{tri_zvezdi}. 
\\

We shall split the proof of Proposition \ref{specialrandom} in several propositions.
The following result will be of importance to pass informations from
a suitable discrete sequence $\varepsilon_N$ to the continuous 
parameter $\varepsilon$. Notice that the independence property of $(g_n)$ is not used in its proof. 
\begin{lemma}\label{Q}
Let $\gamma>0$ be fixed, then
there exists an event $\Sigma_1$ with full measure such that
for every $\omega\in \Sigma_1$ there exists $K>0$ such that
\begin{equation}\label{minkpao}
|g_n(\omega)|< K \langle n\rangle^{\gamma},\quad \forall\, n\in\Z^2\setminus\{0\}.
\end{equation}
\end{lemma}
\begin{proof}
We first prove
\begin{equation}\label{gammafree}
p\Big( \{\omega\in \Omega: \sup_{n\in \Z^2, n\neq 0} \big( \langle n\rangle^{-\gamma} |g_n(\omega)|\big) \geq K \}\Big)
\overset{K\rightarrow \infty}\longrightarrow 0.
\end{equation}
Notice that
\begin{multline*}
p\big( \{\omega\in \Omega:  \sup_{n\in \Z^2, n\neq 0} \big( \langle n\rangle^{-\gamma} |g_n(\omega)|\big) \geq K\} \big)
\\
\leq 
\sum_{ n\in \Z^2, n\neq 0}\,
p\big(\{\omega \in \Omega : |g_n(\omega)| \geq K \langle n\rangle^{\gamma}\}\big).
\end{multline*}
It remains to observe that by gaussianity
$$
p\big(\{\omega\in \Omega : |g_n(\omega)| \geq K \langle n\rangle^{\gamma}\}\big)
<
 C e^{-c \langle n\rangle^{2\gamma} K^2}
 $$
 and we conclude \eqref{gammafree} by elementary considerations.
 
Next we introduce
$$
\Omega_K=\{\omega\in \Omega:  \sup_{n\in \Z^2, n\neq 0} \big( \langle n\rangle^{-\gamma} |g_n(\omega)|\big) < K\}
$$
and
$$\Sigma_1=\bigcup_{K=1}^\infty \Omega_K\,.
$$
Notice that \eqref{minkpao} holds for $\omega\in \Sigma_1$ for a suitable $K$, 
and moreover $\Sigma_1$ has full measure by \eqref{gammafree}.
\end{proof}
\begin{prop}\label{firstprob} Let $s\in (0,1)$ and $q\in (1,\infty)$ be fixed. There exists an event $\tilde \Sigma\subset \Omega$ such that
$p(\tilde \Sigma)=1$ and for every $\omega\in \tilde \Sigma$ there exists $C(\omega)<\infty$ such that:
\begin{equation}\label{quantepsilon}
\|Y_\varepsilon(x,\omega) -Y(x,\omega)\|_{W^{s,q}}<C(\omega) \varepsilon^\kappa
\end{equation}
for a suitable $\kappa>0$.
Moreover we have
\begin{equation}\label{expp}
\|e^{-pY_\varepsilon (x,\omega)}-e^{-pY(x,\omega)}\|_{L^\infty}<C(\omega) \varepsilon^\kappa.\end{equation}
\end{prop}
\begin{proof}
First we notice that by \eqref{quantepsilon} and Sobolev embedding
we have
a.s.
\begin{equation}\label{kapth}\|Y_\varepsilon(x,\omega) -Y(x,\omega)\|_{L^\infty}<C(\omega) \varepsilon^\kappa.
\end{equation}
Hence \eqref{expp} follows from the following computation
$$
|e^{-pY_\varepsilon (x,\omega)}-e^{-pY(x,\omega)}|\lesssim 
|Y_\varepsilon(x,\omega) -Y(x,\omega)| e^{|p| \sup_{\varepsilon} \|Y_\varepsilon (x,\omega)\|_{L^\infty}}$$
and by noticing that by \eqref{kapth} we have 
a.s. $\sup_{\varepsilon} \|Y_\varepsilon (x,\omega)\|_{L^\infty}<\infty$.\\

Next we split the proof of \eqref{quantepsilon} in two steps.\\
\\
{\em First step: proof of \eqref{quantepsilon} for $\varepsilon=\varepsilon_N=N^{-1}$}
\\
\\
For every $p\geq q$ we combine the Minkowski inequality and a standard
bound between the $L^p$ and $L^2$ norms of gaussians
in order to get:
\begin{multline}\label{minkowlas}
\|Y(x,\omega)-Y_{\varepsilon_N}(x,\omega)\|_{L^p(\Omega;W^{s,q})}
= 
\|D^s(Y(x,\omega)-Y_{\varepsilon_N}(x,\omega))\|_{L^p(\Omega;L^{q})}
\\\leq  
\|D^s(Y(x,\omega)-Y_{\varepsilon_N}(x,\omega))\|_{L^{q}(\T^2;L^p(\Omega))}\\
\lesssim \sqrt{p}\|D^s(Y(x,\omega)-Y_{\varepsilon_N}(x,\omega))\|_{L^{q}(\T^2;L^2(\Omega))}
\lesssim \sqrt{p}N^{-1+s}.
\end{multline} 
To justify the last inequality notice that by independence of $g_n(\omega)$ we have for every fixed $x\in \T^2$
the following estimates:
\begin{multline*}\|D^s(Y(x,\omega)-Y_{\varepsilon_N}(x,\omega))\|_{L^2(\Omega)}^2\lesssim
\sum_{|n|\leq N}
\frac{|1- \rho(\varepsilon_N n)|^2}{\langle n\rangle^{4-2s}}+
\sum_{|n|>N}
\frac{|1- \rho(\varepsilon_N n) |^2}{\langle n\rangle^{4-2s}}\\
\lesssim N^{-2}
\sum_{|n|\leq N}
\frac{|n|^2}{\langle n\rangle^{4-2s}}+
\sum_{|n|>N}
\frac{1}{\langle n\rangle^{4-2s}}
\lesssim  N^{-2+2s}
\end{multline*}
where we have used the following consequence of the mean value theorem
$$|1- \rho(\varepsilon_N n)|=|\rho(0)- \rho(\varepsilon_N n)|\lesssim \min\{(\sup |\nabla \rho|) |N^{-1} |n|,
(\sup |\rho|)\}.$$
Next by combining \eqref{minkowlas} and Lemma~4.5 of \cite{tz}, we can write
\begin{equation}\label{BC}
p\big(\{\omega\in \Omega: N^{\frac{1-s}2} \|Y(x,\omega)-
Y_{\varepsilon_N}(x,\omega)\|_{W^{s, q}}> 1\}\big) \leq C\, e^{-c N^{1-s }}
\end{equation} 
where the positive constants $c,C>0$ are independent of $N$. 
The right-hand side of \eqref{BC}  is summable in $N$. Therefore,  we can use the Borel-Cantelli lemma
 to conclude that  there exists a full measure set $\Sigma_2\subset \Omega$
 such that for every $\omega\in \Sigma_2$ there exists $N_0(\omega)\in \N$ such that
 $$
\|Y(x,\omega)-Y_{\varepsilon_N}(x,\omega)\|_{W^{s,q}}\leq N^{\frac{s-1}2}=
\varepsilon_N^{\frac{1-s}{2}},\quad \forall N>N_0(\omega)
$$
and hence 
\begin{equation}\label{discreteversion}
\|Y(x,\omega)-Y_{\varepsilon_N}(x,\omega)\|_{W^{s,q}}<C(\omega) \varepsilon_N^{\frac{1-s}{2}}, \quad \forall N\in \N.
\end{equation}
\\
\\
{\em Second step: proof of \eqref{quantepsilon} for $\varepsilon\in (0,1)$}
\\
\\
Let us set
$
\tilde \Sigma=\Sigma_1\cap \Sigma_2\
$ (where $\Sigma_2$ is given in the first step and $\Sigma_1$ in Lemma \ref{Q}),
then $p(\tilde \Sigma)=1$ and we will show that for every $\omega\in\tilde \Sigma$ we have the desired property.\\
For every $\varepsilon \in (0,1)$  there exists $N\in\N$ such that
\begin{equation}\label{epsilonN}
\varepsilon_{N+1}< \varepsilon\leq \varepsilon_N\
\end{equation}
where $\varepsilon_N=N^{-1}$ is as in the previous step.
We claim that 
\begin{equation}\label{papmam}
\|Y_\varepsilon (x,\omega)-Y_{\varepsilon_N}(x,\omega)\|_{W^{s,q}}
< C(\omega)\varepsilon^{1-s-\gamma}
, \quad \forall \omega\in \Sigma_1
\end{equation}
for every $\gamma>0$. Once the estimate above  is established then
the proof of \eqref{quantepsilon} for $\omega\in \tilde \Sigma$ follows by recalling \eqref{discreteversion}
and the Minkowski inequality:
\begin{multline*}
\|Y_\varepsilon (x,\omega)-Y(x,\omega)\|_{W^{s,q}}
\leq \|Y_\varepsilon (x,\omega)-Y_{\varepsilon_N}(x,\omega)\|_{W^{s,q}}\\
+\|Y_{\varepsilon_N} (x,\omega)-Y(x,\omega)\|_{W^{s,q}}<
C(\omega)\varepsilon^{1-s-\gamma}
 + C(\omega) \varepsilon_N^{\frac{1-s}{2}}\\
 <C(\omega)(\varepsilon^{1-s-\gamma}
 + \varepsilon^{\frac{1-s}{2}}).
 \end{multline*}
Hence we get \eqref{quantepsilon} for $\kappa=\frac{1-s}2$ provided that we choose $\gamma>0$ small enough.\\
 
Next we prove \eqref{papmam}. Due to \eqref{minkpao} and the Minkowski inequality for 
every $\omega\in\Sigma_1$ 
we have some $K>0$ such that:
\begin{equation}\label{timaj}
\|Y_\varepsilon(x,\omega)-Y_{\varepsilon_N}(x,\omega)\|_{W^{s,q}}
\lesssim K
\sum_{n\in\Z^2}
\frac{\big|
\rho\big(\varepsilon n\big) - \rho\big(\varepsilon_N n\big) \big|}{\langle n\rangle^{2-s-\gamma}}\,.
\end{equation}
On the other hand by the mean value theorem and \eqref{epsilonN} we get
\begin{equation*}
\big|
\rho\big(\varepsilon n\big) - \rho\big(\varepsilon_N n\big) \big|
\lesssim |n|N^{-2} \sup_{t\in [0,1]} |\nabla \rho \big((t\varepsilon_N +(1-t) \varepsilon) n\big)|,
\end{equation*}
therefore by using the rapid decay of $\nabla \rho$ we get 
for every $L>0$ the following bound
\begin{equation}\label{meanvaluecomplete}
\big|
\rho\big(\varepsilon n\big) - \rho\big(\varepsilon_N n\big) \big|
\lesssim  |n|N^{-2} \min \{(\sup|\nabla \rho|), |n|^{-L}N^L\}.
\end{equation}
We can now estimate the r.h.s. in \eqref{timaj} as follows:
\begin{multline*}
\sum_{|n|\leq N}
\frac{\big|
\rho\big(\varepsilon n\big) - \rho\big(\varepsilon_N n\big) \big|}{\langle n\rangle^{2-s-\gamma}}
+
\sum_{|n|>N}
\frac{|
\rho (\varepsilon n) - \rho(\varepsilon_N n)|}{\langle n\rangle^{2-s-\gamma}}
\\
\lesssim N^{-2}
\sum_{|n|\leq N}
\frac{1 }{\langle n\rangle^{1-s-\gamma}}+ N^{L-2}\sum_{|n|>N}
\frac{1}{|n|^{L-1}\langle n\rangle^{2-s-\gamma}}
\lesssim N^{-1+s+\gamma}
\end{multline*}
Hence going back to \eqref{timaj}
we get
$$
\|Y_\varepsilon (x,\omega)-Y_{\varepsilon_N}(x,\omega)\|_{W^{s,q}}
< C(\omega)N^{-1+s+\gamma}=C(\omega)\varepsilon_N^{1-s-\gamma}
, \quad \forall \omega\in\Sigma_1
$$
and we conclude \eqref{papmam} since we are assuming \eqref{epsilonN}.
\end{proof}
\begin{prop}\label{secondprob}
Let $q\in (1,\infty)$ be fixed. There exists an event $\tilde \Sigma\subset \Omega$ such that
$p(\tilde \Sigma)=1$ and for every $\omega\in \tilde \Sigma$ there exists $C(\omega)<\infty$ such that:
\begin{equation}\label{quantepsilonsec}
\|\nabla Y_\varepsilon(x,\omega)\|_{L^{q}}<C(\omega) |\ln \varepsilon|
\end{equation}
and
\begin{equation}\label{quantepsilonsecnew}
\|:|\nabla Y_\varepsilon(x,\omega)|^2:\|_{L^{q}}<C(\omega) |\ln \varepsilon|^2.
\end{equation}
\end{prop}

\begin{proof} 
It is easy to deduce \eqref{quantepsilonsecnew} from \eqref{quantepsilonsec}. In fact
we have the following trivial estimate
$$\|:|\nabla Y_\varepsilon(x,\omega)|^2:\|_{L^{q}}
\leq \||\nabla Y_\varepsilon(x,\omega)|\|_{L^{2q}}^2
 +C_\varepsilon< C(\omega) |\ln \varepsilon|^2 +C_\varepsilon$$
 and we conclude by noticing that $C_\varepsilon\lesssim |\ln \varepsilon|$. Next, we focus on the proof of \eqref{quantepsilonsec} that 
we split in two steps.\\
\\
{\em First step: proof of \eqref{quantepsilonsec} for $\varepsilon=\varepsilon_N=N^{-\delta}$, $\delta\in (0,1)$}
\\
\\
Once again, by combining the Minkowski inequality and a standard bound between the $L^p$ and $L^2$ norms of gaussians we get for $p\geq q$:
\begin{multline}\label{minkow}
\|\nabla Y_{\varepsilon_N}(x,\omega)\|_{L^p(\Omega;L^{q})}
\leq  
\|\nabla Y_{\varepsilon_N}(x,\omega)\|_{L^{q}(\T^2;L^p(\Omega))}
\\\lesssim \sqrt{p}\|\nabla Y_{\varepsilon_N}(x, \omega)\|_{L^{q}(\T^2;L^2(\Omega))}
\lesssim \sqrt{p |\ln \varepsilon_N|}
\end{multline} 
The last step follows since, by orthogonality of $g_n(\omega)$, for every $x\in \T^2$ we have:
\begin{multline*}\|\nabla Y_{\varepsilon_N}(x,\omega)\|_{L^{2}(\Omega)}^2\lesssim \sum_{n\in \Z^2, n\neq 0}
\frac{|\rho\big(\varepsilon_N n)|^2}{|n|^2}
\\\lesssim \sum_{n\in \Z^2, 0<|n|<N}
\frac{1}{|n|^2}
+ N^{2\delta L} \sum_{n\in \Z^2, |n|\geq N} \frac1{|n|^{2+2L}}
\lesssim |\ln N| +N^{2L(\delta-1)}
\lesssim |\ln \varepsilon_N|
\end{multline*}
where $L>0$ is any number and we used the fast decay of $\rho$.
Using \eqref{minkow}, we obtain that 
$$
F(\omega)=|\ln \varepsilon_N|^{-1/2}\|\nabla Y_{\varepsilon_N}(x,\omega)\|_{L^{q}}
$$
satisfies 
$
\|F\|_{L^p(\Omega)}\leq C\sqrt{p}
$
and using Lemma~4.5 of \cite{tz} (with $N=1$, according to the notations of \cite{tz}),
we can write
\begin{equation}\label{BCnew}
p\big(\{\omega\in \Omega: |\ln \varepsilon_N|^{-1}\|\nabla
Y_{\varepsilon_N}(x,\omega)\|_{L^{q}}> \lambda\}\big) \leq Ce^{-c|\ln \varepsilon_N| \lambda^2}
\end{equation} 
where the positive constants $c,C>0$ are independent of $N$ and $\lambda>0$ is 
chosen 
large enough in such a way that 
the right-hand side of \eqref{BCnew}  is summable in $N$. Therefore,  we can use the Borel-Cantelli lemma
 to conclude that  there exists a full measure set $\Sigma_2\subset \Omega$
 such that for every $\omega\in \Sigma_2$ there exists $N_0(\omega)\in \N$ such that
$$\|\nabla
Y_{\varepsilon_N}(x,\omega)\|_{L^{q}}\leq \lambda |\ln \varepsilon_N|, \quad \forall N>N_0(\omega)$$
and hence for every $\omega \in \Sigma_2$ we have
$$\|\nabla
Y_{\varepsilon_N}(x,\omega)\|_{L^{q}}\leq C(\omega) |\ln \varepsilon_N|, \quad \forall N\in \N.$$
\\
{\em Second step: proof of \eqref{quantepsilonsec} for $\varepsilon\in (0,1)$}
\\
\\
Exactly as in the proof of Proposition \ref{firstprob}
it is sufficient to estimate 
\begin{equation}\label{lezbiom}
\|Y_{\varepsilon}(x,\omega)-Y_{\varepsilon_N}(x,\omega)\|_{L^{q}}<C(\omega) |\ln \varepsilon|,
\end{equation}
where 
\begin{equation}\label{;;;;}
\varepsilon_{N+1}< \varepsilon \leq \varepsilon_N
\end{equation} with $\varepsilon_N=N^{-\delta}$, provided that 
$\omega$ belongs to the event $\Sigma_1$ given in Lemma \ref{Q} and  $\delta>0$ is small enough .
By Lemma \ref{Q}  we have that  for every $\omega\in \Sigma_1$ there exists a constant $K>0$ such that
\begin{equation}\label{;;;}
\|\nabla Y_{\varepsilon}(x,\omega)-\nabla Y_{\varepsilon_N}(x,\omega)\|_{L^{q}}\lesssim
K
\sum_{n\in\Z^2}
\frac{\big|
\rho\big(\varepsilon n\big) - \rho\big(\varepsilon_N n\big) \big|}{\langle n\rangle^{1-\gamma}}\,.
\end{equation}
Next notice that by combining the mean value theorem  with the strong decay of $\rho$ 
we get for every fixed $L>0$:
\begin{multline*}
\big|
\rho\big(\varepsilon n\big) - \rho\big(\varepsilon_N n\big) \big|
\lesssim |n|N^{-1-\delta} \sup_{t\in [0,1]} |\nabla \rho \big((t\varepsilon_N +(1-t) \varepsilon) n\big)|
\\
\lesssim |n|N^{-1-\delta} \min \{(\sup|\nabla \rho|), |n|^{-L}N^{\delta L}\}.
\end{multline*}
Then we can estimate
\begin{multline*}
\sum_{n\in\Z^2}
\frac{\big|
\rho\big(\varepsilon n\big) - \rho\big(\varepsilon_N n\big) \big|}{\langle n\rangle^{1-\gamma}}
\lesssim N^{-1-\delta} \sum_{n\in \Z^2, 0<|n|<N^\delta}
\langle n\rangle^{\gamma}
+ N^{-1-\delta+\delta L} \sum_{n\in\Z^2, |n|\geq N^\delta}
\langle n\rangle^{\gamma-L}
\\
\lesssim N^{-1-\delta}  N^{\delta(2+\gamma)} + N^{-1-\delta+\delta L} N^{\delta(\gamma - L+2)}
=2N^{-1-\delta}  N^{\delta(2+\gamma)} 
\end{multline*}
and hence  by \eqref{;;;} we get
$$\|Y_{\varepsilon}(x,\omega)-Y_{\varepsilon_N}(x,\omega)\|_{L^{q}}<C(\omega) 
\varepsilon_N^\alpha$$ for a suitable $\alpha>0$, 
provided that we choose $\delta, \gamma>0$ small enough. 
By \eqref{;;;;} we get
$$
\|\nabla Y_{\varepsilon}(x,\omega)- \nabla Y_{\varepsilon_N}(x,\omega)\|_{L^{q}}<C(\omega) 
\varepsilon^\alpha$$ 
and in particular we get \eqref{lezbiom}.

\end{proof}

\begin{prop}\label{firstprobthird} Let $s\in (0,1)$ and $q\in (1,\infty)$ be fixed. There exists an event $\tilde \Sigma\subset \Omega$ such that
$p(\tilde \Sigma)=1$ and for every $\omega\in \tilde \Sigma$ there exists $C(\omega)<\infty$ such that
\begin{equation}\label{quantepsilon:}
\|:|\nabla Y_{\varepsilon}|^2:(x,\omega) -:|\nabla Y|^2:(x,\omega)\|_{W^{-s,q}}<C(\omega) \varepsilon^\kappa
\end{equation}
for a suitable $\kappa>0$. 
\end{prop}
\begin{proof}
We split again the proof of \eqref{quantepsilon:} in two steps.
\\
\\
{\em First step: proof of \eqref{quantepsilon:} for $\varepsilon=\varepsilon_N=N^{-\delta}$, $\delta\in (0,1)$}
\\
\\
Notice that
\begin{multline*}
:|\nabla Y_{\varepsilon_N}|^2:(x,\omega) - :|\nabla Y|^2:(x,\omega)
\\
= \sum_{\substack{(n_1,n_2)\in\Z^4\\n_1\neq 0, n_2\neq 0\\n_1\neq n_2}}\,  
\big[\rho\big (\frac{n_1}{N^\delta}\big )  \rho
\big (\frac{n_2}{N^\delta}\big )
-1\big] 
\frac{n_1\cdot n_2}{|n_1|^2 |n_2|^2}\,
g_{n_1}(\omega)\overline{g_{n_2}(\omega)}\,e^{i(n_1-n_2)\cdot x}
\\+\sum_{n\in\Z^2,n\neq 0}\,
\big[\rho^2(\frac{n}{N^\delta})-1]\frac{|g_n(\omega)|^2 -1}{|n|^2}.
\end{multline*}
For every $p\geq q$, by the Minkowski inequality and 
using hypercontractivity (see \cite{Sim}) to estimate 
the $L^p$ norm of a bilinear form of the gaussian vector $(g_n)$, we get:
\begin{multline}\label{:est}
\|:|\nabla Y_{\varepsilon_N}|^2:(x,\omega) - :|\nabla Y|^2:(x,\omega)\|_{L^p(\Omega;W^{-s,q})}
\\= 
\|D^{-s}(:|\nabla Y_{\varepsilon_N}|^2:(x,\omega) - :|\nabla Y|^2:(x,\omega))\|_{L^p(\Omega;L^{q})}
\\\leq  
\|D^{-s}(:|\nabla Y_{\varepsilon_N}|^2:(x,\omega) - :|\nabla Y|^2:(x,\omega))\|_{L^{q}(\T^2;L^p(\Omega))}\\
\lesssim p\|D^{-s}(:|\nabla Y_{\varepsilon_N}|^2:(x,\omega) - :|\nabla Y|^2:(x,\omega))\|_{L^{q}(\T^2;L^2(\Omega))}.
\end{multline}
Now, we observe that for a complex gaussian $g$ and two nonnegative integers $k_1$, $k_2$, we have that  $\E(g^{k_1}\bar{g}^{k_2})=0$, unless $k_1=k_2$.
Therefore, by  the independence of $(g_n)$, modulo $g_n=\overline{g_{-n}}$, for every fixed $x\in \T^2$, we can write:
\begin{multline*}
\|D^{-s}(:|\nabla Y_{\varepsilon_N}|^2:(x,\omega) - :|\nabla Y|^2:(x,\omega))\|_{L^2(\Omega)}^2\\
\lesssim \sum_{\substack{(n_1,n_2)\in\Z^4\\n_1\neq 0, n_2\neq 0\\n_1\neq n_2}}\,  
\big[\rho\big (\frac{n_1}{N^\delta}\big )  \rho
\big (\frac{n_2}{N^\delta}\big )
-1\big]^2
\frac{|n_1\cdot n_2|^2}{\langle n_1-n_2 \rangle^{2s}|n_1|^4 |n_2|^4}\,
\\+ \sum_{n\in\Z^2,n\neq 0}\,
\big[\rho^2\big (\frac{n}{N^\delta}\big )-1\big]^2\frac{1}{|n|^4}
=I+II,
\end{multline*}
where the implicite constant is independent of $x$. 
Concerning $II$ we notice first that by the mean value theorem 
$$
\big|\rho^2\big (\frac{n}{N^\delta}\big )-1\big |=\big|\rho^2\big (\frac{n}{N^\delta}\big )-\rho^2(0))\big | \lesssim |n| 
N^{-\delta}$$
and by interpolation with the trivial bound
$$\big|\rho^2\big (\frac{n}{N^\delta}\big )-1\big | \lesssim 1$$
we get
$$\big|\rho^2\big (\frac{n}{N^\delta}\big )-1\big | \lesssim |n|^\theta N^{-\delta\theta}, \quad \theta \in (0,1).$$
Then we can estimate 
$$
II\lesssim N^{-2\delta\theta}
\sum_{n\in\Z^2,n\neq 0}\,
\frac{1}{|n|^{4-2\theta}}\lesssim N^{-2\delta\theta} .
$$
Next we estimate $I$.
First notice that 
\begin{equation*}
\rho\big(\frac{n_1}{N^\delta}\big)  \rho
\big ( \frac{n_2}{N^\delta}\big )
-1
=
\rho\big(\frac{n_1}{N^\delta}\big)  \big[\rho
\big ( \frac{n_2}{N^\delta}\big )-\rho(0)\big]
+\big[\rho\big(\frac{n_1}{N^\delta}\big)  
-\rho(0)\big]  \rho
\big (0)
\end{equation*}
and hence by the mean value theorem we get
\begin{equation}\label{meanvalue}
\big| \rho\big(\frac{n_1}{N^\delta}\big)  \rho
\big ( \frac{n_2}{N^\delta}\big )
-1\big|\lesssim (|n_1|+|n_2|) N^{-\delta}.
\end{equation}
that by interpolation with the trivial bound
$$\big| \rho\big(\frac{n_1}{N^\delta}\big)  \rho
\big ( \frac{n_2}{N^\delta}\big )
-1\big|\lesssim 1$$
implies
\begin{equation*}
\big| \rho\big(\frac{n_1}{N^\delta}\big)  \rho
\big ( \frac{n_2}{N^\delta}\big )
-1\big|\lesssim (|n_1|+|n_2|)^\theta N^{-\delta\theta},
\quad \theta\in (0,1).
\end{equation*}
Hence we can evaluate $I$ as follows:
\begin{multline*}
I\lesssim N^{-2\delta\theta}
\sum_{\substack{(n_1,n_2)\in\Z^4\\n_1\neq 0, n_2\neq 0}} \frac{(|n_1|+|n_2|)^{2\theta}}{\langle n_1-n_2 \rangle^{2s}|n_1|^2 |n_2|^2}
\\
\lesssim N^{-2\delta\theta}
\sum_{\substack{(n_1,n_2)\in\Z^4\\n_1\neq 0, n_2\neq 0}} \frac{1}{\langle n_1-n_2 \rangle^{2s}|n_1|^{2-2\theta} 
|n_2|^{2-2\theta}}\lesssim N^{-2\delta\theta},
\end{multline*}
where we have used 
$$\sum_{\substack{(n_1,n_2)\in\Z^4\\n_1\neq 0, n_2\neq 0}} \frac{1}{\langle n_1-n_2 \rangle^{2s}|n_1|^{2-2\theta} 
|n_2|^{2-2\theta}}<\infty$$
which in turn follows from the discrete Young inequality provided that
$2\theta<s$.
Going back to \eqref{:est} we get that for $\delta\in (0,1)$ and $\theta\in (0,\frac s2)$ one has the bound 
\begin{equation*}
\|:|\nabla Y_{\varepsilon_N}|^2:(x,\omega) - :|\nabla Y|^2:(x,\omega)\|_{L^p(\Omega;W^{-s,q})}
\lesssim p N^{-\delta\theta}\,.
\end{equation*}
This estimate,  in conjunction with Lemma~4.5 of \cite{tz} implies
\begin{equation*}
p\big(\{\omega\in \Omega: N^{\frac{\delta \theta}2}
\|:|\nabla Y_{\varepsilon_N}|^2:(x,\omega) - :|\nabla Y|^2:(x,\omega)\|_{W^{-s,q}}> 1\}\big) 
\leq Ce^{-c N^{\frac{\delta \theta}2}}
\end{equation*}  
where the positive constants $c,C>0$ are independent of $N$. 
Since the r.h.s. is summable in $N$ we can apply the Borel-Cantelli lemma and deduce
the existence of an event $\Sigma_2$ with full measure and such that
for every $\omega \in \Sigma_2$ there exists $N_0(\omega)>0$ with the property
$$\|:|\nabla Y_{\varepsilon_N}|^2:(x,\omega) - :|\nabla Y|^2:(x,\omega)\|_{W^{-s,q}}
\leq N^{-\frac{\delta \theta}2}=\varepsilon_N^{\frac \theta 2}, \quad \forall\, N>N_0(\omega)$$
and hence
$$\|:|\nabla Y_{\varepsilon_N}|^2:(x,\omega) - :|\nabla Y|^2:(x,\omega)\|_{W^{-s,q}}
< C(\omega) \varepsilon_N^{\frac{\theta}2}, \quad \forall\, N\in \N.$$
\\
\\
{\em Second step: proof of \eqref{quantepsilon:} for $\varepsilon\in (0,1)$}
\\
\\
We consider a generic $\varepsilon>0$ and we select $N$ in such a way that
$\varepsilon_{N+1}<\varepsilon\leq \varepsilon_N$ where $\varepsilon_N=N^{-\delta}$
as in the first step.
By the Minkowski inequality and the previous step it is sufficient 
to prove
\begin{equation}\label{...}\|:|\nabla Y_{\varepsilon_N}|^2:(x,\omega) - :|\nabla Y_{\varepsilon}|^2:(x,\omega)\|_{W^{-s,q}}
<C(\omega)N^{-\alpha}\end{equation}
for a suitable $\alpha>0$, for $\omega\in \Sigma_1$
where 
$\Sigma_1$ is given in Lemma \ref{Q}.
By Lemma \ref{Q} we deduce that almost surely there exists a finite constant $K>0$
such that \eqref{minkpao} occurs.
Then we have
\begin{multline*}
\|:|\nabla Y_{\varepsilon_N}|^2:(x,\omega) - :|\nabla Y_{\varepsilon}|^2:(x,\omega)\|_{W^{-s,q}}
\\
\lesssim K^2 \sum_{\substack{(n_1,n_2)\in\Z^4\\n_1\neq 0, n_2\neq 0}}
\frac{
\big|\rho(\varepsilon_N n_1)  \rho
(\varepsilon_N n_2)
-\rho(\varepsilon n_1)  \rho
(\varepsilon n_2)\big|}
{\langle n_1-n_2\rangle^s|n_1|^{1-\gamma} |n_2|^{1-\gamma}}.
\end{multline*}
By looking at the argument to prove \eqref{meanvalue} we get:
\begin{equation}\label{meanvaluenew}
\big| \rho(\varepsilon_N n_1)  \rho
( \varepsilon_N n_2)
-\rho(\varepsilon  n_1)  \rho
(\varepsilon n_2)  \big|\lesssim (|n_1|+|n_2|) (\varepsilon_N - \varepsilon)\end{equation}
and hence we can continue the estimate above as follows
\begin{multline}\label{....}
\|:|\nabla Y_{\varepsilon_N}|^2:(x,\omega) - :|\nabla Y_{\varepsilon}|^2:(x,\omega)\|_{W^{-s,q}}
\\
\lesssim 
K^2 (\varepsilon_N - \varepsilon) \sum_{\substack{(n_1,n_2)\in\Z^4\\n_1\neq 0, n_2\neq 0
\\\max\{|n_1|, |n_2|\}\leq N^{2\delta}}}
\frac{|n_1|+|n_2|}{{\langle n_1-n_2\rangle^s|n_1|^{1-\gamma} |n_2|^{1-\gamma}}}
\\
+
K^2  \sum_{\substack{(n_1,n_2)\in\Z^4\\n_1\neq 0, n_2\neq 0
\\\min \{|n_1|, |n_2|\}\geq N^{2\delta}}}
\frac{
\big|\rho(\varepsilon_N n_1)  \rho
(\varepsilon_N n_2)
-\rho(\varepsilon n_1)  \rho
(\varepsilon n_2)\big|}
{\langle n_1-n_2\rangle^s|n_1|^{1-\gamma} |n_2|^{1-\gamma}}
\\
+
K^2  \sum_{\substack{(n_1,n_2)\in\Z^4\\n_1\neq 0, n_2\neq 0
\\\min\{|n_1|, |n_2|\}\leq N^{2\delta}\\
\max\{|n_1|, |n_2|\}\geq N^{2\delta}}}
\frac{
\big|\rho(\varepsilon_N n_1)  \rho
(\varepsilon_N n_2)
-\rho(\varepsilon n_1)  \rho
(\varepsilon n_2)\big|}
{\langle n_1-n_2\rangle^s|n_1|^{1-\gamma} |n_2|^{1-\gamma}}.
\end{multline}
Concerning the first sum on the r.h.s. in \eqref{....} (notice that by symmetry we can assume $|n_1|<|n_2|$) we can estimate as follows:
\begin{multline*}
(\varepsilon_N - \varepsilon) \sum_{\substack{(n_1,n_2)\in\Z^4\\n_1\neq 0, n_2\neq 0
\\|n_1|\leq |n_2|\leq N^{2\delta}}}
\frac{|n_2|}
{\langle n_1-n_2\rangle^s|n_1|^{1-\gamma} |n_2|^{1-\gamma}}
\\
\lesssim N^{-1-\delta} \sum_{\substack{(n_1,n_2)\in\Z^4\\n_1\neq 0, n_2\neq 0
\\|n_1|\leq |n_2|\leq N^{2\delta}}}
|n_2|^\gamma
\lesssim N^{-1-\delta} N^{4\delta} N^{2\delta \gamma}. 
\end{multline*}
For the second sum on the r.h.s. in \eqref{....} we use the fast decay of $\rho$ and hence for every $L>0$ we have
\begin{multline*}
\sum_{\substack{(n_1,n_2)\in\Z^4\\n_1\neq 0, n_2\neq 0
\\\min \{|n_1|, |n_2|\}\geq N^{2\delta}}}
\frac{
\big|\rho(\varepsilon_N n_1)  \rho
(\varepsilon_N n_2)
-\rho(\varepsilon n_1)  \rho
(\varepsilon n_2)\big|}
{\langle n_1-n_2\rangle^s|n_1|^{1-\gamma} |n_2|^{1-\gamma}}
\\\lesssim N^{2\delta L}\sum_{\substack{(n_1,n_2)\in\Z^4\\n_1\neq 0, n_2\neq 0
\\\min \{|n_1|, |n_2|\}\geq N^{2\delta}}}
\frac{1}
{|n_1|^{1-\gamma+L} |n_2|^{1-\gamma+L}}
\\\lesssim N^{2\delta L} N^{{2\delta}(\gamma -L+1)}N^{{2\delta}(\gamma -L+1)}
\lesssim N^{4\delta \gamma-2\delta L+4\delta}.
\end{multline*}
By using again the fast decay of $\rho$ we can estimate the third sum on the r.h.s. in \eqref{....}
as follows (we can assume by symmetry $|n_1|<N^{2\delta}\leq |n_2|$):
\begin{multline*}\sum_{\substack{(n_1,n_2)\in\Z^4\\n_1\neq 0, n_2\neq 0
\\|n_1|<N^{2\delta}\leq |n_2|}}
\frac{
\big|\rho(\varepsilon_N n_1)  \rho
(\varepsilon_N n_2)
-\rho(\varepsilon n_1)  \rho
(\varepsilon n_2)\big|}
{\langle n_1-n_2\rangle^s|n_1|^{1-\gamma} |n_2|^{1-\gamma}}
\\
\lesssim {N^{\delta L}}
\sum_{\substack{(n_1,n_2)\in\Z^4\\n_1\neq 0, n_2\neq 0
\\|n_1|<N^{2\delta}\leq |n_2|}}
\frac{1}
{|n_1|^{1-\gamma} |n_2|^{1-\gamma+L}}
\lesssim {N^{\delta L}} N^{2\delta(\gamma+1)} N^{2\delta(\gamma -L+1)}
\lesssim N^{-\delta L+4\delta \gamma+4\delta }.
\end{multline*}
The proof of \eqref{...} is now complete provided we choose $L$ large and $\delta,\gamma$ small enough.
\end{proof}
We complete this section by noticing that the analysis we performed here may allow the extension of the results of \cite{OPT} to a continuous family of approximation problems. 
\section{Some useful facts}
Next we provide a result that will be useful in the sequel. The proof is inspired by \cite{DW}, nevertheless we provide a proof for the sake of completeness.
From a technical viewpoint the minor difference is that our proof involves Sobolev spaces, while the one in \cite{DW} uses  H\"older spaces.
\begin{prop}\label{ercol10} Let $v_\varepsilon(t,x, \omega)$ be as in Theorem \ref{main}. Then
there exists an event $\Sigma\subset \Omega$ such that
$p(\Sigma)=1$, $\Sigma\subset \Sigma_0$ where $\Sigma_0$ is the event in Proposition 
\ref{specialrandom} and for every $\omega \in \Sigma$ there exists a finite constant $C(\omega)>0$ such that:
\begin{equation}\label{ercol2}
\sup_{\varepsilon\in (0,1),  t\in\R} \|v_\varepsilon(t,x,\omega)\|_{H^1}<C(\omega)
\end{equation}
and
\begin{equation}\label{H2}
\|e^{Y(x,\omega)}u_0(x)\|_{H^2}<C(\omega).
\end{equation}
\end{prop}
\begin{proof}
We introduce the event $\Sigma$ as follows
\begin{equation}\label{defsigma}
\Sigma=\{\omega\in \Omega: e^{Y(x,\omega)}u_0(x)\in H^2\} \cap \Sigma_0
\end{equation}
where $\Sigma_0$ is the event provided by Proposition~\ref{specialrandom}.\\

We now state the fundamental conservation laws satisfied by $v_{\varepsilon}$.  We have the mass conservation 
\begin{equation}\label{mass}
\frac d{dt}  \int_{\T^2} e^{-2Y_\varepsilon} |v_\varepsilon (t)|^2 =0
\end{equation}
and the energy conservation
\begin{equation}\label{hamiltonian}
\frac d{dt} \Big( \int_{\T^2} e^{-2Y_\varepsilon } 
\big(|\nabla v_\varepsilon (t)|^2 
-:|\nabla Y_\varepsilon|^2: |v_\varepsilon (t)|^2 -\frac {2\lambda}{p+2}  e^{-pY_\varepsilon}  |v_\varepsilon(t)|^{p+2}
\big)
\Big) =0.
\end{equation}
Of course, the conservation laws \eqref{mass} and \eqref{hamiltonian} give the key global information in our analysis. 
By using \eqref{mass}  and  Propositions~\ref{specialrandom},  we get
\begin{multline}\label{L2}
\int_{\T^2} |v_\varepsilon (t)|^2 
\leq \|e^{2Y_\varepsilon}\|_{L^\infty}\int_{\T^2} e^{-2Y_\varepsilon} |v_\varepsilon (t)|^2 
< C(\omega) \int_{\T^2} e^{-2Y_\varepsilon} |v(0)|^2 
\\
< C(\omega) \|e^{-2Y_\varepsilon}\|_{L^\infty} \int_{\T^2} |v(0)|^2 
<C(\omega) \int_{\T^2} |v(0)|^2 .
\end{multline}
In order to control $\|\nabla v_\varepsilon (t) \|_{L^2}$ we first notice that
by duality and by Lemma 2.2 in \cite{OPT} (see also \cite{GKO} and the references therein), and by using Proposition
\ref{specialrandom} we get for $s\in (0,1)$:
\begin{multline*}\big |\int_{\T^2} e^{-2Y_\varepsilon} :|\nabla Y_\varepsilon|^2: |v_\varepsilon (t)|^2 \big |
\leq \|:|\nabla Y_\varepsilon|^2: \|_{W^{-s, q}} \| e^{-2Y_\varepsilon} |v_\varepsilon (t)|^2 \|_{W^{s,q'}}
\\
< C(\omega) (\|e^{-2Y_\varepsilon}\|_{W^{s,2}} 
\||v_\varepsilon (t)|^2\|_{L^r} + \|e^{-2Y_\varepsilon}\|_{L^{l}}
\||v_\varepsilon (t)|^2\|_{W^{s,\frac 32}})
\\<C(\omega) 
(\|v_\varepsilon (t) \|_{L^{2r}}^2 + \|v_\varepsilon (t)\|_{L^m}\|v_\varepsilon (t)\|_{H^s}),
\end{multline*}
where $\frac 1{q'}=\frac 12 + \frac 1{r}=\frac 2{3} + \frac 1{l}$, $\frac 23=\frac 1m+ \frac 12$ and $q<\infty$ is large enough. 
We now fix $s=\frac{1}{2}$. By using now interpolation and the Sobolev embedding we get
\begin{eqnarray*}
\big |\int_{\T^2} e^{-2Y_\varepsilon} :|\nabla Y_\varepsilon|^2: |v_\varepsilon (t)|^2\big |
& <  & C(\omega)
\|v_\varepsilon (t) \|_{H^{\frac{1}{2}}} \|v_\varepsilon (t) \|_{H^{1}}
\\
& < & 
 C(\omega)   \|v_\varepsilon (t) \|_{L^{2}}^{\frac{1}{2}}  \|v_\varepsilon (t) \|_{H^{1}}^{\frac{3}{2}}\,.
 \end{eqnarray*} 
By combining this estimate with  \eqref{hamiltonian} (and by using that we are assuming $\lambda\leq 0$) we get
\begin{multline*}\int_{\T^2} e^{-2Y_\varepsilon} |\nabla v_\varepsilon (t)|^2 
\lesssim \| e^{-Y_\varepsilon} \nabla v_\varepsilon (0)\|_{L^2}^2
\\+
C(\omega)   (\|v_\varepsilon (t) \|_{L^{2}}^{\frac{1}{2}}  \|v_\varepsilon (t) \|_{H^{1}}^{\frac{3}{2}}
+\|v_\varepsilon (0) \|_{L^{2}}^{\frac{1}{2}}  \|v_\varepsilon (0) \|_{H^{1}}^{\frac{3}{2}})
+   \|e^{-Y_\varepsilon} v_\varepsilon(0)\|_{L^{p+2}}^{p+2}
\end{multline*}
which in turn by interpolation, Sobolev embedding and \eqref{L2} implies
\begin{multline*}
\int_{\T^2} ( |\nabla v_\varepsilon (t)|^2 + |v_\varepsilon (t)|^2) 
\lesssim  C(\omega) 
[{\mathcal P}(\|v_\varepsilon (0) \|_{L^{2}})  \|v_\varepsilon (t) \|_{H^{1}}^{\frac{3}{2}}
+{\mathcal P}(\|v_\varepsilon (0) \|_{H^{1}})]
\end{multline*}
where ${\mathcal P}$ denotes a polynomial function and we have used Proposition~\ref{specialrandom} to estimate a.s.  $\sup_{\varepsilon\in (0,1)} \|e^{-Y_\varepsilon}\|_{L^\infty}<C(\omega)$. 
We therefore have the bound 
$$\int_{\T^2} ( |\nabla v_\varepsilon (t)|^2 + |v_\varepsilon (t)|^2) 
\lesssim C(\omega)  \|v_\varepsilon (t) \|_{H^{1}}^{\frac{3}{2}}+C(\omega),
$$
where the random constant $C(\omega)$ is finite for every $\omega\in \Sigma$. We conclude the proof  by the classical Young inequality. 
\end{proof}
In the sequel we shall need suitable versions of the Gronwall lemma. Although they are very classical
we prefer to state them, in particular we emphasize how the estimates depend 
from the constants involved. We also mention that the estimates below are implicitely used in \cite{DW},
however for the sake of clarity we prefer to give below the precise statements that we need.
\begin{prop}\label{gronwlog}
Let $f(t)$ be a  non-negative real valued function such that  for $t\in [0, \infty)$:
$$f(t) \leq A + B \int_0^t f(s) \ln (C+ f(s))ds,$$
where $A, B, C\in (1, \infty)$.
Then we have the following upper bound
$$f(t)\leq (A+C)^{e^{Bt}}\,. $$
\end{prop}
\begin{proof} Notice that by assumption
\begin{equation}\label{gghhkk}
C+f(t)\leq A +C+ B \int_0^t (C+f(s)) \ln (C+ f(s))ds:=F(t)\,.
\end{equation}
Therefore
$$
F'(t)= B (C+f(t))\ln (C+f(t))\leq B F(t)\ln F(t)\,.
$$
Hence 
$$
\frac{d}{dt}\big(
\ln \ln (F(t))
\big)\leq B
$$
which implies after integration between $0$ and $t$ :
$$
\ln\ln (F(t))\leq \ln\ln(A+C)+Bt\,.
$$
By taking twice the exponential we obtain
$$F(t)\leq (A+C)^{e^{Bt}}.$$
Coming back to \eqref{gghhkk}, we get the needed bound. 
\end{proof}
\begin{prop}\label{gronwclass}
Let $f(t)$ be a  non-negative real valued function for $t\in [0, \infty)$, such that:
$$f'(t)\leq A +B f(t), \quad
f(0)=0$$
where $A, B\in (0, \infty)$.
Then we have the following upper bound:
$$f(t)\leq  AB^{-1} e^{Bt}.$$
\end{prop}

\begin{proof}
We notice that
$\frac d{dt}(e^{-Bt} f(t))\leq Ae^{-Bt}$
and hence
$$f(t)\leq A e^{Bt} \int_0^t e^{-Bs} ds=\frac{A}{B} e^{Bt} (1-e^{-Bt})\leq \frac{A}{B} e^{Bt}.$$
\end{proof}
\section{Modified energy for the gauged NLS on $\T^2$ and $H^2$ a-priori bounds}
In the sequel $v_\varepsilon(t,x,\omega)$ will denote the unique solution to:
\begin{multline}\label{NLSgauge}
i\partial_t v_\varepsilon  = \Delta v_\varepsilon  - 2\nabla v_\varepsilon  \cdot \nabla Y_\varepsilon + v_\varepsilon :|\nabla Y_\varepsilon |^2:
+\lambda  e^{-pY_\varepsilon}  v_\varepsilon|v_\varepsilon|^p,\,\,
\\
v_\varepsilon(0,x)=v_0(x)\in H^2,
\end{multline}
where $\lambda\in \R$, $p\geq 2$.
\begin{prop}\label{modifen} 
We have the identity
$$
\frac d{dt} ({\mathcal E}_{\varepsilon}(v_\varepsilon )[t])= \lambda \mathcal H_{\varepsilon}(v_\varepsilon)[t],
$$
where
\begin{equation*}
{\mathcal E}_{\varepsilon}(v_\varepsilon)[t]={\mathcal F}_{\varepsilon}(v_\varepsilon)[t]
+ \lambda{\mathcal G}_{\varepsilon}(v_\varepsilon)[t]
\end{equation*}
and the energies ${\mathcal F}_{\varepsilon}, {\mathcal G}_{\varepsilon}, \mathcal H_{\varepsilon}$
are defined as follows on a generic time dependent function
$w(t,x)$. The kinetic energy is defined by
\begin{eqnarray*}
{\mathcal F}_{\varepsilon}(w)[t] = \int_{\T^2} |\Delta w(t)|^2 e^{-2Y_\varepsilon}
-4 \Re\int_{\T^2} \Delta w(t) \nabla Y_\varepsilon \cdot \nabla \bar w(t) e^{-2Y_\varepsilon}
\\
+4 \sum_{i=1}^2 \int_{\T^2}  (|\partial_i w(t)|^2) (\partial_i Y_\varepsilon)^2 e^{-2Y_\varepsilon}
+8  \Re \int_{\T^2}  \partial_1 Y_\varepsilon \partial_2 Y_\varepsilon  \partial_1 w(t) \partial_2 \bar w(t) e^{-2Y_\varepsilon}
\\
+2 \Re \int_{\T^2} w(t) :|\nabla Y_\varepsilon|^2:  \nabla \bar w(t) \cdot \nabla (e^{-2Y_\varepsilon})
\\
+2 \Re \int_{\T^2} \Delta w(t) \bar w(t) :|\nabla Y_\varepsilon|^2: e^{-2Y_\varepsilon}
+ \int_{\T^2}  |w(t)|^2 (:|\nabla Y_\varepsilon|^2:)^2 e^{-2Y_\varepsilon}\,\,.
\end{eqnarray*}
The potential energy is defined by
\begin{multline*}
{\mathcal G}_{\varepsilon}(w)[t] = 
- \int_{\T^2} |\nabla w(t)|^2|w(t)|^p e^{-(p+2)Y_\varepsilon}
-2 \Re \int_{\T^2} w(t) \nabla (|w(t)|^p) \cdot \nabla \bar w(t)  e^{-(p+2)Y_\varepsilon}
\\
+\frac p4 \int_{\T^2} |\nabla(|w(t)|^2)|^2|w(t)|^{p-2}  e^{-(p+2)Y_\varepsilon}
\\
+\frac {2 }{p+2} \int_{\T^2}  |w(t)|^{p+2} :|\nabla Y_\varepsilon|^2:  e^{-(p+2)Y_\varepsilon}\,\,
+2p\Re \int_{\T^2} w(t) |w(t)|^p \nabla Y_{\varepsilon}  \cdot \nabla \bar w(t) e^{-(p+2)Y_{\varepsilon}}\,.
\end{multline*}
Finally, the lack of exact conservation is measured by the functional 
\begin{multline*}
\mathcal H_{\varepsilon}(w)[t] = 
-\int_{\T^2} |\nabla w(t)|^2 \partial_t (|w(t)|^p)  e^{-(p+2)Y_\varepsilon}
\\
-2\Re \int_{\T^2} \partial_t w(t)\nabla (|w(t)|^p) \cdot \nabla \bar w (t) 
 e^{-(p+2)Y_\varepsilon}
\\
-\frac p4 \int_{\T^2} |\nabla(|w(t)|^2)|^2\partial_t (|w(t)|^{p-2})
e^{-(p+2)Y_\varepsilon} \\
+2p\Re \int_{\T^2}\partial_t (w(t) |w(t)|^p) 
\nabla Y_\varepsilon \cdot \nabla \bar w(t) e^{-(p+2)Y_\varepsilon}\, . 
\end{multline*}
\end{prop}
\begin{remark}
Notice that in the linear case (namely \eqref{NLSgauge} with $\lambda=0$) we get the following exact conservation law:
$$\frac d{dt} {\mathcal F}_{\varepsilon}(v_\varepsilon )=0.$$
\end{remark}
The proof of Proposition~\ref{modifen} will be presented in the last section of the paper.  Next, we estimate ${\mathcal H}_\varepsilon(v_\varepsilon)$ and  the lower order terms in the energy ${\mathcal E}_\varepsilon(v_\varepsilon)$.
They will play a crucial role in order to get the key $H^2$ a-priori bound
for $v_\varepsilon$.
In the sequel we shall assume that $v_\varepsilon$ solves \eqref{NLSgauge} with $\lambda\leq 0$
and $p\in [2,3]$. In particular we are allowed to use Proposition \ref{ercol10} in order to control
a.s. $\|v_\varepsilon(t,x)\|_{H^1}$ uniformly w.r.t. $\varepsilon$ and $t$.
\begin{prop}\label{firstcor}
Let $\Sigma\subset \Omega$ be the event of full probability, obtained in Propositions~\ref{ercol10}. 
Then there exists a random variable $C(\omega)$ finite on $\Sigma$ such that for every $\varepsilon\in (0,\frac{1}{2})$:
\begin{equation*}
|{\mathcal H}_\varepsilon(v_\varepsilon)|<
C(\omega) \Big( \|e^{-Y_\varepsilon}\Delta v_\varepsilon\|_{L^2}^2
\ln^\frac{p-1}2 \big (2+
\|e^{-Y_\varepsilon} \Delta v_\varepsilon\|_{L^2}\big )
+  |\ln \varepsilon|^8 \Big).
\end{equation*}
\end{prop}
\begin{proof}
By using the H\"older inequality, the Leibnitz rule and the diamagnetic inequality $|\partial_t |u||\leq |\partial_t u|$ we get that the first three terms in ${\mathcal H}_\varepsilon(v_\varepsilon)$ can be estimated by:
\begin{equation*} \int_{\T^2} |\partial_t v_\varepsilon ||\nabla v_\varepsilon|^2 |v_\varepsilon|^{p-1} 
e^{-(p+2)Y_\varepsilon}
\leq \|\partial_t v_\varepsilon \|_{L^2(e^{-(p+2)Y_\varepsilon})} 
\|\nabla v_\varepsilon\|_{L^4({e^{-(p+2)Y_\varepsilon})}}^2 \|v_\varepsilon\|_{L^\infty}^{p-1}
\end{equation*}
that by the Brezis-Gallou\"et inequality (see \cite{BG}) can be estimates as follows:
\begin{multline*}
\dots \lesssim 
\|\partial_t v_\varepsilon \|_{L^2(e^{-(p+2)Y_\varepsilon})} \|\nabla v_\varepsilon\|_{L^4(e^{-(p+2)Y_\varepsilon})}^2
\|v_\varepsilon\|_{H^1}^{p-1}\\\times  \ln^\frac{p-1}2 (2+ \|v_\varepsilon\|_{L^2} + \|\Delta v_\varepsilon\|_{L^2})
\end{multline*}
and by using the equation solved by $v_\varepsilon(t,x)$:
\begin{multline*}\dots 
\lesssim
\|\Delta v_\varepsilon \|_{L^2(e^{-(p+2)Y_\varepsilon})} \|\nabla v_\varepsilon\|_{L^4(e^{-(p+2)Y_\varepsilon})}^2
\|v_\varepsilon\|_{H^1}^{p-1} \ln^\frac{p-1}2 (2+ \|v_\varepsilon\|_{L^2} + \|\Delta v_\varepsilon\|_{L^2})\\
+\|\nabla v_\varepsilon  \cdot \nabla Y_\varepsilon \|_{L^2({e^{-(p+2)Y_\varepsilon})}} 
\|\nabla v_\varepsilon\|_{L^4(e^{-(p+2)Y_\varepsilon})  }^2
\|v_\varepsilon\|_{H^1}^{p-1} \ln^\frac{p-1}2 (2+ \|v_\varepsilon\|_{L^2} + \|\Delta v_\varepsilon\|_{L^2})
\\
+\|v_\varepsilon :|\nabla Y_\varepsilon |^2: \|_{L^2(e^{-(p+2)Y_\varepsilon})} \|\nabla v_\varepsilon\|_{L^4(e^{-(p+2)Y_\varepsilon})}^2
\|v_\varepsilon\|_{H^1}^{p-1} \ln^\frac{p-1}2 (2+ \|v_\varepsilon\|_{L^2} + \|\Delta v_\varepsilon\|_{L^2})\\
+\|e^{-pY_\varepsilon}v_\varepsilon |v_\varepsilon|^p\|_{L^2(e^{-(p+2)Y_\varepsilon})} \|\nabla v_\varepsilon\|_{L^4(e^{-(p+2)Y_\varepsilon})}^2
\|v_\varepsilon\|_{H^1}^{p-1} \ln^\frac{p-1}2 (2+ \|v_\varepsilon\|_{L^2} + \|\Delta v_\varepsilon\|_{L^2})
\\
=I+II+III+IV.\end{multline*}
Next we recall a family of estimates that will be useful to control $I,II, III, IV$.
We shall also use without any further comment Propositions \ref{specialrandom} and \ref{ercol10}. 
We have the Gagliardo-Nirenberg type inequality
\begin{equation}\label{GN}
\|\nabla u\|_{L^4(\T^2)}^2\leq C \|\Delta u \|_{L^2(\T^2)}\|\nabla u\|_{L^2(\T^2)} \,. 
\end{equation}
Indeed, using the Sobolev embedding $H^{\frac{1}{2}}(\T^2)\subset L^4(\T^2)$, we can write
$$
\|\nabla u\|_{L^4(\T^2)}^2
\leq C \|\nabla u \|^2_{H^{\frac{1}{2}}(\T^2)}\leq
C \|\nabla u\|_{H^1(\T^2)}\|\nabla u\|_{L^2(\T^2)}\,.
$$
It remains to observe that 
$$
\|\nabla u\|_{H^1(\T^2)}\leq C\|\Delta u\|_{L^2(\T^2)}\,.
$$
Therefore we have \eqref{GN}. Now, using \eqref{GN} we get:
\begin{equation*}
\|\nabla v_\varepsilon\|_{L^4(e^{-(p+2)Y_\varepsilon})}^2
< C(\omega)
\|\nabla v_\varepsilon\|_{L^4}^2
< C(\omega)\|\Delta v_\varepsilon\|_{L^2}
\|\nabla v_\varepsilon\|_{L^2} \end{equation*}
and also
\begin{multline*}
\|\nabla v_\varepsilon  \cdot \nabla Y_\varepsilon \|_{L^2(e^{-(p+2)Y_\varepsilon})}
<  C(\omega) \|\nabla v_\varepsilon\|_{L^4}\|\nabla Y_\varepsilon \|_{L^4}\\
< C(\omega)\|\Delta v_\varepsilon\|_{L^2}^\frac 12
\|\nabla v_\varepsilon\|_{L^2}^\frac 12 \|\nabla Y_\varepsilon \|_{L^4}.\end{multline*}
Next notice that
\begin{multline*}\|v_\varepsilon :|\nabla Y_\varepsilon |^2: \|_{L^2(e^{-(p+2)Y_\varepsilon})}<
C(\omega)
\|v_\varepsilon \|_{L^4}\|:|\nabla Y_\varepsilon |^2: \|_{L^4} \\
<C(\omega)
\|v_\varepsilon \|_{H^1} \|:|\nabla Y_\varepsilon |^2: \|_{L^4}\,,
\end{multline*}
where we have used the Sobolev embedding.
Again by the Sobolev embedding we get:
\begin{equation*}
\| e^{-pY_\varepsilon} v_\varepsilon |v_\varepsilon|^p\|_{L^2(e^{-(p+2)Y_\varepsilon})}<C(\omega) \|v_\varepsilon\|_{L^{2(p+1)}}^{p+1}<C(\omega) \|v_\varepsilon\|_{H^1}^{p+1}.
\end{equation*}
Finally notice that
\begin{equation*}\label{ercol}
\|\Delta v_\varepsilon\|_{L^2}\leq \|e^{-Y_\varepsilon}\Delta v_\varepsilon\|_{L^2} \|e^{Y_\varepsilon}\|_{L^\infty}
<C(\omega) \|e^{-Y_\varepsilon}\Delta v_\varepsilon\|_{L^2} .
\end{equation*}
Based on the estimates above we get:
\begin{equation*}
I< C(\omega) \|e^{-Y_\varepsilon}\Delta v_\varepsilon\|_{L^2}^2
\ln^\frac{p-1}2 \big(2 +  C(\omega)+C(\omega) \|e^{-Y_\varepsilon} \Delta v_\varepsilon\|_{L^2}\big)\,.
\end{equation*}
and also
\begin{multline*}
II< C(\omega) 
\|e^{-Y_\varepsilon}\Delta v_\varepsilon\|_{L^2}^\frac 32 |\ln \varepsilon|
\ln^\frac{p-1}2  (2 + C(\omega)+ C(\omega)\|e^{-Y_\varepsilon} \Delta v_\varepsilon\|_{L^2})
\\
<
C(\omega)\|e^{-Y_\varepsilon}\Delta v_\varepsilon\|_{L^2}^2
+|\ln \varepsilon|^4
\ln^{2(p-1)} (2+ C(\omega)+C(\omega)\|e^{-Y_\varepsilon} \Delta v_\varepsilon\|_{L^2})
\\< 
C(\omega)  (\|e^{-Y_\varepsilon}\Delta v_\varepsilon\|_{L^2}^2+1)
+|\ln \varepsilon|^8,
\end{multline*}
where we used the Young inequality.
We conclude with the following estimates:
\begin{multline*}
III< C(\omega) 
\|e^{-Y_\varepsilon}\Delta v_\varepsilon\|_{L^2} |\ln \varepsilon|^{2}
\ln^\frac{p-1}2 (2+ C(\omega)+ C(\omega) \|e^{-Y_\varepsilon} \Delta v_\varepsilon\|_{L^2})\\
< C(\omega) 
(\|e^{-Y_\varepsilon}\Delta v_\varepsilon\|_{L^2}^2 + 1)
+|\ln \varepsilon|^8 \,
\end{multline*}
and
\begin{multline*}
IV< C(\omega) 
\|e^{-Y_\varepsilon}\Delta v_\varepsilon\|_{L^2}
\ln^\frac{p-1}2  (2+  C(\omega)+C(\omega)\|e^{-Y_\varepsilon} \Delta v_\varepsilon\|_{L^2})
\\<  C(\omega) 
(\|e^{-Y_\varepsilon}\Delta v_\varepsilon\|_{L^2}^2 +1). 
\end{multline*}
Summarizing we can control the first three terms in ${\mathcal H}_\varepsilon(v_\varepsilon)$.
Concerning 
the last term in the expression of ${\mathcal H}_\varepsilon (v_\varepsilon)$ we can estimate it as follows:
\begin{multline*}
\int_{\T^2} |\partial_t v_\varepsilon | |v_\varepsilon |^p 
|\nabla Y_\varepsilon| |\nabla v_\varepsilon| e^{-(p+2)Y_\varepsilon}\\
\leq \|e^{-(p+2)Y_\varepsilon}\|_{L^\infty} \|v_\varepsilon \|_{L^{8p}}^p
\|\partial_t v_\varepsilon \|_{L^2} \|\nabla Y_\varepsilon\|_{L^8} 
 \|\nabla  v_\varepsilon\|_{L^4} 
\\
<C(\omega) |\ln \varepsilon| \|\partial_t v_\varepsilon \|_{L^2} 
 \|\nabla v_\varepsilon \|_{L^4} 
\end{multline*}
where we have used Propositions  \ref{specialrandom} and \ref{ercol10}
in conjunction with the Sobolev embedding to control $\|v_\varepsilon \|_{L^{8p}}$. Hence by the Gagliardo-Nirenberg \eqref{GN}  inequality 
and by using the equation solved by $v_\varepsilon$ we can continue as follows 
\begin{multline*}\dots < C(\omega) |\ln \varepsilon| \|\Delta v_\varepsilon \|_{L^2} 
\|\Delta v_\varepsilon\|_{L^2}^\frac 12
\|\nabla v_\varepsilon\|_{L^2}^\frac 12\\
+C(\omega) |\ln \varepsilon| 
\|\nabla v_\varepsilon  \cdot \nabla Y_\varepsilon \|_{L^2} \|\Delta v_\varepsilon\|_{L^2}^\frac 12 
\|\nabla v_\varepsilon\|_{L^2}^\frac 12\\
+C(\omega) |\ln \varepsilon| \|v_\varepsilon :|\nabla Y_\varepsilon |^2: \|_{L^2}
\|\Delta v_\varepsilon\|_{L^2}^\frac 12
\|\nabla v_\varepsilon\|_{L^2}^\frac 12
\\+C(\omega) |\ln \varepsilon| \|e^{-pY_\varepsilon}v_\varepsilon |v_\varepsilon|^p\|_{L^2}
\|\Delta v_\varepsilon\|_{L^2}^\frac 12
\|\nabla v_\varepsilon\|_{L^2}^\frac 12.
\end{multline*}
and by the Sobolev embedding 
\begin{multline*}\dots <C(\omega) |\ln \varepsilon| 
\|\Delta v_\varepsilon\|_{L^2}^\frac 32 
+C(\omega) |\ln \varepsilon| 
\|\nabla v_\varepsilon\|_{L^4} \|\nabla Y_\varepsilon \|_{L^4} \|\Delta v_\varepsilon \|_{L^2}^\frac 12 \\
+C(\omega) |\ln \varepsilon| \|v_\varepsilon\|_{L^4} \|:|\nabla Y_\varepsilon |^2: \|_{L^4}
\|\Delta v_\varepsilon\|_{L^2}^\frac 12
+C(\omega) |\ln \varepsilon| \|v_\varepsilon |v_\varepsilon|^p\|_{L^2}
\|\Delta v_\varepsilon\|_{L^2}^\frac 12
\\
< C(\omega) |\ln \varepsilon| 
\|\Delta v_\varepsilon\|_{L^2}^\frac 32 
+C(\omega) |\ln \varepsilon|^2  \|\Delta v_\varepsilon \|_{L^2} 
+C(\omega) |\ln \varepsilon|^3 
\|\Delta v_\varepsilon\|_{L^2}^\frac 12
+C(\omega) |\ln \varepsilon| 
\|\Delta v_\varepsilon\|_{L^2}^\frac 12.
\end{multline*}
The conclusion is now straighforward.
\end{proof}
\begin{prop}\label{bounds} 
Let $\Sigma\subset \Omega$ be the event of full probability, obtained in Propositions~\ref{ercol10}. 
For every $\delta>0$ and $\omega\in \Sigma$ there exists a finite constant  $C(\omega, \delta)>0$ such that for every $\varepsilon\in (0,\frac{1}{2})$:
\begin{equation}\label{mart}
|{\mathcal F}_{\varepsilon}(v_\varepsilon)-\int_{\T^2} |\Delta v_\varepsilon|^2 e^ {-2Y_\varepsilon}
|< \delta \|e^{-Y_\varepsilon} \Delta v_\varepsilon\|_{L^2}^2
+
C(\omega, \delta) |\ln \varepsilon|^4
\end{equation}
and
\begin{equation}\label{pastor}|{\mathcal G}_{\varepsilon}(v_\varepsilon)|<
\delta \|e^{-Y_\varepsilon} \Delta v_\varepsilon\|_{L^2}^2 +C(\omega, \delta)  |\ln \varepsilon|^4.
\end{equation}
\end{prop}
\begin{proof}
We estimate the terms involved in the expression
${\mathcal F}_{\varepsilon}(v_\varepsilon)-\int_{\T^2} |\Delta v_\varepsilon|^2 e^ {-2Y_\varepsilon}$.
Since the arguments are quite similar to the ones used along of Proposition~\ref{firstcor}, we skip the details. Using  Propositions ~\ref{specialrandom} and \ref{ercol10}, we can write
\begin{multline*}|\int_{\T^2} \Delta v_\varepsilon \nabla Y_\varepsilon \cdot \nabla \bar v_\varepsilon e^{-2Y_\varepsilon}|\leq
\|\Delta v_\varepsilon\|_{L^2} \|\nabla Y_\varepsilon\|_{L^4} \|\nabla v_\varepsilon\|_{L^4}  
\|e^{-2Y_\varepsilon}\|_{L^\infty}\\
< C(\omega) \|\Delta v_\varepsilon\|_{L^2}   |\ln \varepsilon| 
\|\Delta v_\varepsilon\|_{L^2}^\frac 12 \| \nabla v_\varepsilon\|_{L^2}^\frac 12
< C(\omega)  \|e^{-Y_\varepsilon} \Delta v_\varepsilon\|_{L^2}^\frac 32   |\ln \varepsilon|\\
<  \delta \|e^{-Y_\varepsilon} \Delta v_\varepsilon\|_{L^2}^2  + C(\omega, \delta)  |\ln \varepsilon|^4.
\end{multline*}
Next notice that third and fourth term in the energy ${\mathcal F}_{\varepsilon}$ can be estimated by 
\begin{multline*}
 \|\nabla v_\varepsilon\|_{L^4}^2 \|\nabla Y_\varepsilon\|_{L^4}^2 \|e^{-2Y_\varepsilon}\|_{L^\infty}
\\ < C(\omega)|\ln \varepsilon|^2 
 \|\Delta v_\varepsilon\|_{L^2} \| \nabla v_\varepsilon\|_{L^2}
 < C(\omega)  |\ln \varepsilon|^2   \|e^{-Y_\varepsilon} \Delta v_\varepsilon\|_{L^2} 
\\ < \delta \|e^{-Y_\varepsilon} \Delta v_\varepsilon\|_{L^2}^2  + C(\omega, \delta) |\ln \varepsilon|^4.
\end{multline*}
By similar arguments and Sobolev embedding we get:
\begin{multline*}
| \int_{\T^2} v_\varepsilon :|\nabla Y_\varepsilon|^2:  \nabla \bar v_\varepsilon \cdot \nabla (e^{-2Y_\varepsilon})|
\\
\leq 2 \|v_\varepsilon\|_{L^8} \|:|\nabla Y_\varepsilon|^2:\|_{L^8} 
\|\nabla v_\varepsilon\|_{L^2} \|\nabla Y_\varepsilon\|_{L^4} \|e^{-2Y_\varepsilon}\|_{L^\infty}
<  C(\omega) 
|\ln \varepsilon|^{3}.
\end{multline*}
Next we estimate the other term to be controlled:
\begin{multline*}
|\int_{\T^2} \Delta v_\varepsilon \bar v_\varepsilon :|\nabla Y_\varepsilon|^2: e^{-2Y_\varepsilon}|
\leq
\|\Delta v_\varepsilon\|_{L^2} \|:|\nabla Y_\varepsilon|^2:\|_{L^4} \|v_\varepsilon\|_{L^4}  
\|e^{-2Y_\varepsilon}\|_{L^\infty}
\\
<  C(\omega) 
\|e^{-Y_\varepsilon}\Delta v_\varepsilon\|_{L^2} |\ln \varepsilon|^2 
< \delta \|e^{-Y_\varepsilon} \Delta v_\varepsilon\|_{L^2}^2  
+ C(\omega, \delta)   |\ln \varepsilon|^4\,,
\end{multline*}
where we have used the Sobolev embedding.
We conclude the proof of \eqref{mart} by the following estimates:
\begin{eqnarray*}
|\int_{\T^2}  |v_\varepsilon|^2 (:|\nabla Y_\varepsilon|^2:)^2 e^{-2Y_\varepsilon}|
 \leq \|v_\varepsilon\|_{L^4}^2 \|:|\nabla Y_\varepsilon|^2:\|_{L^4}^2 \|e^{-2Y_\varepsilon}\|_{L^\infty}
< C(\omega)  |\ln \varepsilon|^4 \,,
\end{eqnarray*}
where we have used again the Sobolev embedding.
Next, we prove \eqref{pastor}. The first, second and third terms in the definition
of ${\mathcal G}_\varepsilon (v_\varepsilon)$ can be estimated
essentially by the same argument. Let us focus on the first one:
\begin{multline*}
\int_{\T^2} |\nabla v_\varepsilon|^2|v_\varepsilon|^p {e^{-(p+2)Y_\varepsilon}}
\leq \|\nabla v_\varepsilon\|_{L^4} ^2 \| v_\varepsilon\|_{L^{2p}} ^p \|e^{-(p+2)Y_\varepsilon}\|_{L^\infty}
\\
<C(\omega)\|\Delta v_\varepsilon\|_{L^2} \|\nabla v_\varepsilon\|_{L^2} \| v_\varepsilon\|_{H^1} ^p
< C(\omega)  \|e^{-Y_\varepsilon}\Delta v_\varepsilon\|_{L^2}\,,
\end{multline*}
where we have used again the Sobolev embedding, the Gagliardo-Nirenberg
inequality and  Propositions \ref{specialrandom} and \ref{ercol10}.
Concerning the fourth term in the definition
of ${\mathcal G}_\varepsilon (v_\varepsilon)$ we get by the H\"older inequality and the Sobolev embedding
\begin{multline*} \int_{\T^2}  |v_\varepsilon|^{p+2} :|\nabla Y_\varepsilon|^2:  e^{-(p+2)Y_\varepsilon}
\\\leq \|:|\nabla Y_\varepsilon|^2:\|_{L^4}  \|v_\varepsilon\|^{p+2}_{L^{\frac 43(p+2)}} 
\|e^{-(p+2)Y_\varepsilon}\|_{L^\infty} <C(\omega)  |\ln \varepsilon|^2
\end{multline*}
where we have used again 
Propositions \ref{specialrandom} and \ref{ercol10}.
Finally we focus on the last term in the definition
of ${\mathcal G}_\varepsilon (v_\varepsilon)$ that can be estimated as follows
\begin{multline*}
\int_{\T^2} |v_\varepsilon|^{p+1} |\nabla Y_{\varepsilon}| |\nabla v_\varepsilon|  e^{-(p+2)Y_{\varepsilon}}
\\\leq \| v_\varepsilon\|_{L^{4(p+1)}}^{p+1}\|\nabla Y_{\varepsilon}\|_{L^4} \|\nabla v_\varepsilon\|_{L^2}  \|e^{-(p+2)Y_{\varepsilon}}\|_{L^\infty}
<C(\omega) |\ln \varepsilon|
\end{multline*}
where we have used Propositions~\ref{specialrandom} and \ref{ercol10}.
\end{proof}
As already mentioned in the introduction  we carefully follow the approach in \cite{DW} along the proof of Theorem~\ref{main}. The main novelty being the following
$H^2$ a-priori bound that we extend to the regime of the nonlinearity $2\leq p\leq 3$.
Next we shall focus on the proof of the following Proposition (to be compared with Proposition~4.2 in \cite{DW})
which is the most important result of this section.
\begin{prop}\label{prop123} 
Let $\Sigma\subset \Omega$ be the event of full probability, obtained in Propositions~\ref{ercol10} and let $T>0$ be fixed. 
Then there exists a random variable $C(\omega)>0$ finite for every $\omega\in 
\Sigma$ and such that for every $\varepsilon\in (0,\frac{1}{2})$, 
$$\sup_{t\in [-T,T]}\|v_{\varepsilon}(t,x)\|_{H^2}< |\ln \varepsilon|^{C(\omega)}.$$
\end{prop}
\begin{proof}[Proof of Proposition~\ref{prop123}]
We only consider positive times $t$. The case $t<0$ can be treated similarly. 
We shall prove the following estimate
$$
\|e^{-Y_\varepsilon} \Delta v_\varepsilon(t)\|_{L^2}^2\leq
 \Big(C(\omega) (1+ |\ln \varepsilon|^4
+ T|\ln \varepsilon|^8)\Big)^{e^{C(\omega) t}}, \quad \forall t\in [0,T]
$$
for a suitable random constant which is finite a.s.,
then the conclusion follows  by 
$$\|\Delta v_\varepsilon(t)\|_{L^2}^2\leq \|e^{-Y_\varepsilon} \Delta v_\varepsilon(t)\|_{L^2}^2
\|e^{2 Y_\varepsilon}\|_{L^\infty}<C(\omega)  \|e^{-Y_\varepsilon} \Delta v_\varepsilon(t)\|_{L^2}^2.$$
By Proposition \ref{modifen} after integration in time
and by using Propositions \ref{firstcor} and \ref{bounds} (where we choose $\delta>0$ small 
in such a way that we can absorb on the l.h.s. the factor 
$\|e^{-Y_\varepsilon}\Delta v_\varepsilon(t)\|_{L^2}^2$)
we can write:
\begin{multline}\label{eleute}\|e^{-Y_\varepsilon}\Delta v_\varepsilon(t)\|_{L^2}^2
\\
\leq
C(\omega)
\int_0^t  \big[ \|e^{-Y_\varepsilon}\Delta v_\varepsilon\|_{L^2}^2
\ln^\frac{p-1}2 (2+ \|e^{-Y_\varepsilon} \Delta v_\varepsilon\|_{L^2})
+  |\ln \varepsilon|^8 \big]
\\
+ C(\omega)    |\ln \varepsilon|^4+{\mathcal E}_\varepsilon (v_0)\,.
\end{multline}
Notice also that by Proposition \ref{bounds} one can can show the following bound
for every $\omega$ belonging to the event given in Proposition 
\ref{ercol10}: 
$$
|{\mathcal E}_\varepsilon (v_0)|< C(\omega)(1 + |\ln \varepsilon|^4)\,.
$$
Hence, by recalling that $\frac{p-1}2\leq 1$, we deduce from \eqref{eleute} the following  bound
\begin{multline*}\|e^{-Y_\varepsilon}\Delta v_\varepsilon(t)\|_{L^2}^2
\leq
C(\omega)
\int_0^t  \big[ \|e^{-Y_\varepsilon}\Delta v_\varepsilon\|_{L^2}^2
\ln (2+ \|e^{-Y_\varepsilon} \Delta v_\varepsilon\|_{L^2})
\big]  ds\\
+C(\omega)(1 + |\ln \varepsilon|^4 +T  |\ln \varepsilon|^8), \quad \forall t\in (0,T).
\end{multline*}
We can apply Proposition~\ref{gronwlog} and the conclusion follows.

\end{proof}
\section{Proof of Theorem \ref{main}}
\subsection{Convergence of the approximate solutions}
\begin{prop} 
Let $T>0$ be fixed and $v_\varepsilon(t,x, \omega)$ be as in Theorem \ref{main}. Then 
there exists $ v(t,x, \omega)\in {\mathcal C}(\R;H^\gamma)$ such that 
$$\sup_{t\in [-T, T]} \|v_\varepsilon(t,x, \omega) -  v(t,x, \omega)\|_{H^\gamma}
\overset{\varepsilon \rightarrow 0} \longrightarrow 0, \quad a.s. \quad w.r.t. \quad \omega.$$
\end{prop}
\begin{proof} 
We shall only consider positive times, the analysis for negative times being similar.
Let us fix $T>0$.  Set 
$$
r(t,x)=v_{\varepsilon_1}(t,x)-v_{\varepsilon_2}(t,x), \quad \varepsilon_2\geq \varepsilon_1\,.
$$
Then the  equation solved by $r$ is the following one:
\begin{multline*}i \partial_t r= \Delta r - 2\nabla r  \cdot \nabla Y_{\varepsilon_1} 
+ r :|\nabla Y_{\varepsilon_1} |^2: 
\\- 2\nabla v_{\varepsilon_2}   \cdot \nabla (Y_{\varepsilon_1}  - Y_{\varepsilon_2})
+ v_{\varepsilon_2} (:|\nabla Y_{\varepsilon_1} |^2:-:|\nabla Y_{\varepsilon_2} |^2:)
\\+ \lambda e^{-pY_{\varepsilon_1}} r |v_{\varepsilon_1}|^p  - \lambda e^{-pY_{\varepsilon_1}} v_{\varepsilon_2} 
(|v_{\varepsilon_2}|^p -|v_{\varepsilon_1}|^p)+\lambda v_{\varepsilon_2}|v_{\varepsilon_2}|^p
( e^{-pY_{\varepsilon_1}} -e^{-pY_{\varepsilon_2}} ).
\end{multline*}
We multiply the equation by $e^{-2Y_{\varepsilon_1}(x)}\bar r(t,x)$ and we consider the imaginary part, then we get
\begin{multline} \frac 12 \frac d{dt} \int_{\T^2} e^{-2Y_{\varepsilon_1}} |r(t)|^2 
= - 2 \Im \int_{\T^2} e^{-2Y_{\varepsilon_1}} \bar r (t) 
\nabla v_{\varepsilon_2}   \cdot \nabla (Y_{\varepsilon_1}  - Y_{\varepsilon_2}) \\
+ \Im \int_{\T^2} e^{-2Y_{\varepsilon_1}} \bar r(t) v_{\varepsilon_2}(t) (:|\nabla Y_{\varepsilon_1} |^2:-:|\nabla Y_{\varepsilon_2} |^2:)
\\
+\lambda \Im \int_{\T^2} e^{-(p+2)Y_{\varepsilon_1}} 
|r(t)|^2 |v_{\varepsilon_1}(t)|^p
\\
-\lambda \Im \int_{\T^2}  
e^{-(p+2)Y_{\varepsilon_1}} \bar r(t)v_{\varepsilon_2}(t) 
(|v_{\varepsilon_2}(t)|^p -|v_{\varepsilon_1}(t)|^p)
\\
+\lambda \Im \int_{\T^2} e^{-2Y_{\varepsilon_1}}\bar r(t)v_{\varepsilon_2}(t)|v_{\varepsilon_2}(t)|^p
( e^{-pY_{\varepsilon_1}} -e^{-pY_{\varepsilon_2})}
\\=I+II+III+IV+V.
\end{multline}
From now on we choose $\omega \in \Sigma$ where the event $\Sigma$ is defined as in Proposition~\ref{ercol10}.
We estimate $I$ by using duality and Lemma 2.2 in \cite{OPT} (see also the proof of Proposition~\ref{ercol10}):
\begin{multline*}
|I|\lesssim \|\nabla (Y_{\varepsilon_1}  - Y_{\varepsilon_2})\|_{W^{-s,q}}
\| e^{-2Y_{\varepsilon_1}} \bar r (t) 
\nabla v_{\varepsilon_2} (t)\|_{W^{s,q'}}
\\\lesssim \|\nabla (Y_{\varepsilon_1}  - Y_{\varepsilon_2})\|_{W^{-s,q}}
\times \big (\|e^{-2Y_{\varepsilon_1}}\|_{W^{s,q_1}} \|r(t)\|_{L^{q_2}}
\|\nabla v_{\varepsilon_2}(t)\|_{L^{q_3}}\\
+\|e^{-2Y_{\varepsilon_1}}\|_{L^{q_1}} \|r(t)\|_{W^{s,q_2}}
\|\nabla v_{\varepsilon_2}(t)\|_{L^{q_3}}+\|e^{-2Y_{\varepsilon_1}}\|_{L^{q_1}} \|r(t)\|_{L^{q_2}}
\|\nabla v_{\varepsilon_2}(t)\|_{W^{s,q_3}}\big ) 
 \end{multline*}
where $\frac 1{q'}=\frac 1{q_1}+\frac 1{q_2}+\frac 1{q_3}$ and $s\in (0,1), q\in (1, \infty)$.
Next notice that by choosing $s\in (0,1)$ small enough, by using Sobolev embedding 
and by recalling
Propositions~\ref{specialrandom}, \ref{ercol10} and \ref{prop123}  we get
$$|I|< C(\omega) \varepsilon_2^\kappa \|v_{\varepsilon_2}(t)\|_{H^2} \|r(t)\|_{H^1}
< C(\omega) \varepsilon_2^\kappa |\ln \varepsilon_2|^{C(\omega)}.$$
By a similar argument we can estimate $II$ as follows:
\begin{multline*}
|II|\lesssim \|:|\nabla Y_{\varepsilon_1}|^2:  - :|Y_{\varepsilon_2}|^2:\|_{W^{-s,q}}
\| e^{-2Y_{\varepsilon_1}} \bar r (t) 
v_{\varepsilon_2}(t) \|_{W^{s,q'}}
\\\lesssim \|\nabla (Y_{\varepsilon_1}  - Y_{\varepsilon_2})\|_{W^{-s,q}}
\times \big (\|e^{-2Y_{\varepsilon_1}}\|_{W^{s,q_1}} \|r(t)\|_{L^{q_2}}
\|v_{\varepsilon_2}(t)\|_{L^{q_3}}\\
+\|e^{-2Y_{\varepsilon_1}}\|_{L^{q_1}} \|r(t)\|_{W^{s,q_2}}
\|v_{\varepsilon_2}(t)\|_{L^{q_3}}+\|e^{-2Y_{\varepsilon_1}}\|_{L^{q_1}} \|r(t)\|_{L^{q_2}}
\|v_{\varepsilon_2}(t)\|_{W^{s,q_3}}\big ) 
 \end{multline*}
and hence by using Sobolev embedding 
and by recalling
Propositions~\ref{specialrandom}, \ref{ercol10} and \ref{prop123}  we get for $s$ small enough
$$|II|\lesssim C(\omega) \varepsilon_2^\kappa \|v_{\varepsilon_2}(t)\|_{H^2}<
C(\omega) \varepsilon_2^\kappa.
$$
The estimate of the term $III$ is rather classical and can be done by using the Br\'ezis-Gallou\"et inequality
(see \cite{BG}). More precisely we get:
\begin{multline*}
|III|\lesssim \|e^{-\frac{p+2}2Y_{\varepsilon_1}} r(t)\|_{L^2}^2 \|v_{\varepsilon_1}(t)\|_{L^\infty}^p
\\
\lesssim 
\|e^{-\frac{p+2}2 Y_{\varepsilon_1}} r(t)\|_{L^2}^2 (\|v_{\varepsilon_1}(t)\|_{H^1}^p
\ln^\frac p2  (2+  \|v_{\varepsilon_1}(t)\|_{H^2})
\\
< C(\omega) \|e^{-\frac{p+2}2 Y_{\varepsilon_1}} r(t)\|_{L^2}^2 
\ln^\frac p2  (2+  \|v_{\varepsilon_1}(t)\|_{H^2})
\end{multline*}
where we have used at the last step Proposition \ref{ercol10}.
In order to control $\|v_{\varepsilon_1}(t)\|_{H^2}$ 
we use Proposition \ref{prop123} and we get
\begin{multline*}
|III|< C(\omega) \|e^{-\frac{p+2}2 Y_{\varepsilon_1}} r(t)\|_{L^2}^2 
\ln^\frac p2 (|\ln \varepsilon_1|^{C(\omega)})\\
<
C(\omega) \|e^{-Y_{\varepsilon_1}} r(t)\|_{L^2}^2 
\ln^\frac p2 (|\ln \varepsilon_1|^{C(\omega)})
\end{multline*}
Next, arguing as in the estimate of $III$, we get by combining
Propositions \ref{ercol10} and \ref{prop123}
\begin{multline*}
|IV|\lesssim \|e^{-\frac{p+2}2Y_{\varepsilon_1}} r(t)\|_{L^2}^2 (\|v_{\varepsilon_1}(t)\|_{L^\infty}^{p-1}
+\|v_{\varepsilon_2}(t)\|_{L^\infty}^{p-1})\|v_{\varepsilon_2}(t)\|_{L^\infty}
\\
\lesssim 
\|e^{-\frac{p+2}2Y_{\varepsilon_1}} r(t)\|_{L^2}^2 \ln^\frac p2 (|\ln \varepsilon_1|^{C(\omega)})
<C(\omega) 
\|e^{-Y_{\varepsilon_1}} r(t)\|_{L^2}^2 \ln^\frac p2 (|\ln \varepsilon_1|^{C(\omega)}).
\end{multline*}
Finally by the H\"older inequality, Propositions
\ref{specialrandom} and \ref{ercol10} we estimate
\begin{equation*}
 |V|\lesssim 
\|e^{-2Y_{\varepsilon_1}}\|_{L^\infty} \|\bar r(t)\|_{L^2}
\|v_{\varepsilon_2}(t)\|_{L^{2(p+1)}}^{p+1}
\| e^{-pY_{\varepsilon_1}} -e^{-pY_{\varepsilon_2}}\|_{L^\infty}
\\
< C(\omega)\varepsilon_2^\kappa.
\end{equation*}
Summarizing we obtain 
\begin{multline}\label{GrOnW}
\frac 12 \frac d{dt} \int_{\T^2} e^{-2Y_{\varepsilon_1}} |r(t)|^2 
< C(\omega) \varepsilon_2^\kappa |\ln \varepsilon_2|^{C(\omega)}
\\ +  C(\omega) 
\ln^\frac p2  (|\ln \varepsilon_1|^{C(\omega)}
)  \int_{\T^2} e^{-2Y_{\varepsilon_1}} |r(t)|^2  \,.
\end{multline}
Next we split the proof in two steps.
\\
\\
{\em First step: 
$v_{2^{-k}}(t,x,\omega)\overset {k\rightarrow \infty} \longrightarrow  v(t,x,\omega)$ 
for every $\omega \in \Sigma$.}
\\
\\
We consider $r=v_{2^{-(k+1)}} - v_{2^{-k}}$. 
Then 
by combining Proposition~\ref{gronwclass} and \eqref{GrOnW} (where we choose $\varepsilon_1
=2^{-(k+1)}$ and $
\varepsilon_2=2^{-k}$) we get:
\begin{multline*}
\sup_{t\in (0,T)}\int_{\T^2} e^{-2Y_{2^{-(k+1)}}}  |v_{2^{-(k+1)}}(t)- 
v_{2^{-k}}(t)|^2 \\
< \frac{ C(\omega)  2^{-k\kappa} |\ln 2^{-k}|^{C(\omega)} 
e^{C(\omega) \ln^\frac p2  (|\ln 2^{-(k+1)}|^{C(\omega)}) T}}{\ln^\frac p2  (|\ln 2^{-(k+1)}|^{C(\omega)})}
\end{multline*}
By recalling that for every
$\omega\in \Sigma$ we have $\sup_k \|e^{2Y_{2^{-(k+1)}}}\|_{L^\infty}<C(\omega)<\infty$
we deduce that 
the bound above implies
\begin{multline*}
\sup_{t\in (0,T)}\int_{\T^2} |v_{2^{-(k+1)}}(t)- 
v_{2^{-k}}(t)|^2
< \frac{ C(\omega) 2^{-k\kappa} |\ln 2^{-k}|^{C(\omega)} 
e^{C(\omega) \ln^\frac p2  (|\ln 2^{-(k+1)}|^{C(\omega)}) T}}{\ln^\frac p2  (|\ln 2^{-(k+1)}|^{C(\omega)})}.
\end{multline*}
By combining this estimate with interpolation and with Proposition \ref{prop123} 
we deduce for every $\gamma\in [0,2)$ the following the bound 
\begin{multline*}
\sup_{t\in (0,T)}\|v_{2^{-(k+1)}}(t)- 
v_{2^{-k}}(t)\|_{H^\gamma} \\
< \frac{ C(\omega) 2^{-k \tilde \kappa} |\ln 2^{-k}|^{C(\omega)} 
e^{C(\omega) \ln^\frac p 2  (|\ln 2^{-(k+1)}|^{C(\omega)}) T}}{\ln^\frac {\tilde p}2  (|\ln 2^{-(k+1)}|^{C(\omega)})}.
\end{multline*}
where $\tilde \kappa, \tilde p>0$ are constants that depend from 
the interpolation inequality. It is easy to check that
$$\sum_k  \frac{ C(\omega) 2^{-k \tilde \kappa} |\ln 2^{-k}|^{C(\omega)} 
e^{C(\omega) \ln^\frac p 2  (|\ln 2^{-(k+1)}|^{C(\omega)}) T}}{\ln^\frac {\tilde p}2  (|\ln 2^{-(k+1)}|^{C(\omega)})}
<\infty$$
and therefore  $(v_{2^{-k}})$ is a Cauchy sequence in ${\mathcal C}([0,T];H^\gamma)$ and we conclude.
\\
\\
{\em Second step: $v_{\varepsilon}(t,x,\omega)\overset {\varepsilon\rightarrow 0} \longrightarrow 
 v(t,x,\omega)$ for every $\omega\in \Sigma$.}
\\
\\
For every $\varepsilon\in (2^{-(k+1)}, 2^{-k})$ we introduce
$r=v_{\varepsilon}- v_{2^{-k}}$. Then by combining \eqref{GrOnW}
(where we choose $\varepsilon_1=\varepsilon$ and $\varepsilon_2=2^{-k}$) with Proposition \ref{gronwclass}
and arguing as above  we get
\begin{equation*}
\sup_{t\in (0,T)}\|v_{\varepsilon}(t)- v_{2^{-k}}(t)\|_{H^\gamma} < \frac{ C(\omega) 2^{-k \tilde \kappa} |\ln 2^{-k}|^{C(\omega)} 
e^{C(\omega) \ln^\frac p 2  (|\ln \varepsilon|^{C(\omega)}) T}}{\ln^\frac {\tilde p}2  (|\ln \varepsilon |^{C(\omega)})}
\end{equation*}
and hence (recall that $\varepsilon\in (2^{-(k+1)}, 2^{-k})$)
$$\sup_{\varepsilon\in (2^{-(k+1)}, 2^{-k})}\|v_{\varepsilon}(t)- 
v_{2^{-k}}(t)\|_{H^\gamma}\overset {k\rightarrow \infty} \longrightarrow 0.$$
We conclude by recalling the first step. 
\end{proof}
\subsection{Uniqueness for \eqref{NLSgaugeintronewbar}}
It follows from the analysis of the previous section that $v_{\varepsilon}$ converges almost surely to a solution of \eqref{NLSgaugeintronewbar}.
We next prove the uniqueness of this solution. 
\begin{prop}\label{unIQU} 
Let $\Sigma\subset \Omega$ be the full measure event defined in Proposition~\ref{ercol10} and $T>0$.
For every $\omega \in \Sigma$ 
there exists at most one solution 
$ v(t,x)\in {\mathcal C}([0,T];H^\gamma)$ to \eqref{NLSgaugeintronewbar}
for $\gamma>1$.
\end{prop}
\begin{proof}
Assume $ v_1(t,x)$ and  $ v_2(t,x)$ are two solutions, then we consider the difference
$r(t,x)= v_1(t,x)- v_2(t,x)$ which solves
\begin{equation*}
\begin{cases}
i\partial_t r  = \Delta r  - 2\nabla r  \cdot \nabla Y + r :|\nabla Y |^2:
+\lambda  e^{-pY}( v_1| v_1|^p- v_2| v_2|^p),\\
r(0,x)=0.
\end{cases}\end{equation*}
Next we multiply the equation by $e^{-2Y_\varepsilon(x)} \bar r(t,x)$ where $\varepsilon\in (0,1)$, we integrate by parts
and we take the imaginary part, finally we get:
\begin{multline*}
\frac 12 \frac d{dt} \int_{\T^2} e^{-2Y_\varepsilon} |r(t)|^2=
2 \int_{\T^2} e^{-2Y_\varepsilon} \bar r(t)\nabla r(t) \cdot \nabla (Y_\varepsilon - Y) \\
+\lambda \Im \int_{\T^2} e^{-2Y_\varepsilon-pY}  \bar r(t)
( v_1(t)| v_1(t)|^p- v_2(t)| v_2(t)|^p)=I+II.
\end{multline*}
By the Sobolev embedding $H^\gamma\subset L^\infty$ we get
\begin{multline*}
II\leq \|e^{-pY}\|_{L^\infty} \| e^{-Y_\varepsilon}r(t)\|_{L^2}^2 \sup_{t\in [0,T]} (\|v_1(t)\|_{H^\gamma}^p+\|v_2(t)\|_{H^\gamma}^p)\\
< C(\omega) \| e^{-Y_\varepsilon}r(t)\|_{L^2}^2.\end{multline*}
For the term $I$ we get by duality and Lemma 2.2 in \cite{OPT} (see the proof of Proposition~ \ref{ercol10} for more details) the following estimate
\begin{multline*}
|I|\leq \|\nabla (Y_\varepsilon - Y)\|_{W^{-s,q}} 
\|e^{-2Y_\varepsilon} \bar r(t)\nabla r(t)\|_{W^{s,q'}}
\\< C(\omega) \varepsilon^\kappa 
(\|e^{-2Y_\varepsilon}\|_{L^{q_1}} \|\bar r(t)\|_{L^{q_2}}\|\nabla r(t)\|_{W^{s,2}}
+\|e^{-2Y_\varepsilon}\|_{W^{s,q_1}} \|\bar r(t)\|_{L^{q_2}}\|\nabla r(t)\|_{L^{2}}
\\+\|e^{-2Y_\varepsilon}\|_{L^{q_1}} \|\bar r(t)\|_{W^{s,q_2}} \|\nabla r(t)\|_{L^{2}}  ) 
\end{multline*}
where $s\in (0,1), q\in (0, \infty)$, $\frac 1{q'}=\frac 1{q_1}+
\frac 1{q_2}+\frac 1{2}$ and we have used Proposition \ref{specialrandom} at the second step.
By Sobolev embedding, provided that we choose $s$ small enough, and Proposition \ref{specialrandom} one can show that 
\begin{equation*}
|I|< C(\omega) \varepsilon^\kappa \sup_{t\in [0,T]} (\|v_1(t)\|_{H^\gamma}^{2}+\|v_2(t)\|_{H^\gamma}^{2})
< C(\omega) \varepsilon^\kappa.
\end{equation*}
Summarizing we get
$$\frac d{dt} \int_{\T^2} e^{-2Y_\varepsilon} |r(t)|^2<C(\omega) (\int_{\T^2} e^{-2Y_\varepsilon} |r(t)|^2
+\varepsilon^\kappa), \quad r(0)=0.
$$
We deduce by Proposition \ref{gronwclass} that
$$\int_{\T^2} e^{-2Y_\varepsilon} |r(t)|^2<C(\omega) \varepsilon^\kappa e^{C(\omega)t}$$
and hence by passing to the limit $\varepsilon\rightarrow 0$ we deduce
$\int_{\T^2} e^{-2Y} |r(t)|^2=0$.
\end{proof}
\section{Proof of Theorem \ref{corproof}}
The proof of \eqref{eas}  follows by combining the transformation \eqref{gaugetranf}
with Theorem~\ref{main}. Since now we shall denote by $\Sigma\subset \Omega$ 
the event of full probability given by the intersection
of the ones defined in Theorem~\ref{main} and in Proposition \ref{ercol10}.  In order to prove \eqref{diffnew} we first show
\begin{equation}\label{diff}
\sup_{t\in [-T,T]} \big \|e^{Y_\varepsilon(x,\omega)} |u_\varepsilon(t,x,\omega)| - | v (t,x, \omega)| \big \|_
{{H^\gamma(\T^2)}\cap L^\infty (\T^2)} \overset{\varepsilon\rightarrow 0}
\longrightarrow 0, \quad \forall\, \omega \in \Sigma.
\end{equation}
Notice that from \eqref{eas} and the Sobolev embedding, we get 
$$\sup_{t\in [-T,T]} \|e^{-iC_\varepsilon t} e^{Y_\varepsilon} u_\varepsilon(t)- v(t)\|_{L^2\cap L^\infty}
\overset{\varepsilon\rightarrow 0} \longrightarrow 0, \quad \forall \,\omega\in \Sigma$$
and hence by the triangle inequality in $\C$, 
\begin{equation}\label{L2Linf}\sup_{t\in [-T,T]} \|e^{Y_\varepsilon} |u_\varepsilon(t)|-| v(t)|\|_{L^2\cap L^\infty}
\overset{\varepsilon\rightarrow 0} \longrightarrow 0,  \quad \forall \,\omega\in \Sigma.
\end{equation}
Next we prove 
\begin{equation}\label{L2Linfremcat}\sup_{t\in [-T,T]} \|e^{Y_{\varepsilon}} |u_{\varepsilon}(t)|-| v(t)|\|_{H^\gamma}
\overset{\varepsilon\rightarrow 0} \longrightarrow 0, \quad \forall\, \omega\in \Sigma, \quad \gamma\in (0,1).
\end{equation}
Since
\begin{equation}\label{tinadoc}
\sup_{t\in [-T,T]} \|v_{\varepsilon}(t) -  v(t)\|_{H^\gamma}\overset {\varepsilon\rightarrow 0} \longrightarrow 0,
\quad \forall \,\omega\in \Sigma,
\quad \gamma\in [0,2)
\end{equation}
we get in particular 
\begin{equation}\label{preego}\sup_{t\in [-T,T]} 
\| v(t)\|_{H^1} <\infty, \quad \forall\, \omega\in \Sigma 
\end{equation}
and hence by the diamagnetic inequality
\begin{equation}\label{postego}\sup_{t\in [-T,T]} 
\|| v(t)|\|_{H^1} <\infty, \quad \forall\, \omega\in \Sigma.
\end{equation}
On the other hand by \eqref{tinadoc} we have
\begin{equation}\label{prepostego}
\sup_{\substack{\varepsilon\in (0,1), t\in [-T,T]}} \|v_{\varepsilon}(t)\|_{H^\gamma} <\infty, \quad \forall\, \omega\in \Sigma, \quad \gamma\in [0, 2).
\end{equation}
Next notice that $e^{Y_{\varepsilon}} |u_{\varepsilon}(t)|=|v_{\varepsilon}(t)|$
and hence by the diamagnetic inequality $\|e^{Y_{\varepsilon}} |u_{\varepsilon}(t)|\|_{H^1}\leq \|v_{\varepsilon}(t)\|_{H^1}$ and hence summarizing
\begin{equation}\label{unifboun}\sup_{\substack{\varepsilon\in (0,1), t\in [-T,T]}} 
\big(\max\big \{ \|e^{Y_{\varepsilon}} |u_{\varepsilon}(t)|\|_{H^1}, \|| v(t)|\|_{H^1}\big \}\big)<\infty,
\quad \quad \forall\, \omega\in \Sigma.
\end{equation}
By interpolation between  the uniform bound \eqref{unifboun} and 
\eqref{L2Linf} we get \eqref{L2Linfremcat}.
\\

Finally we prove \eqref{diffnew}.  We show first the following fact
\begin{equation}\label{fiffaltern}\sup_{t\in [-T,T]} \|e^{Y_\varepsilon} |u_\varepsilon(t)| - e^{Y} |u_\varepsilon(t)|\|_{H^\gamma
\cap L^\infty}
\overset{\varepsilon\rightarrow 0 }\longrightarrow 0, \quad \forall\, \omega\in \Sigma,
\quad \gamma\in (0,1)\
\end{equation}
which in turn implies by \eqref{diff} the following convergence
\begin{equation}\label{diffaltenew}
\sup_{t\in [-T,T]} \big \|e^{Y} |u_\varepsilon(t)| - | v(t)| \big \|_
{{H^\gamma}\cap L^\infty} \overset{\varepsilon\rightarrow 0}
\longrightarrow 0, \quad \forall \,\omega\in \Sigma, \quad \gamma\in (0,1).
\end{equation}
We shall establish the following equivalent form of \eqref{fiffaltern}:
\begin{equation}\label{almfinjes}
\|e^{Y}(1-e^{Y_{\varepsilon}-Y}) |u_{\varepsilon}(t)|\big \|_{{H^\gamma}\cap L^\infty}
\overset{\varepsilon\rightarrow 0}\longrightarrow 0, \quad \forall \,\omega \in \Sigma, \quad \gamma\in (0,1).
\end{equation}
We first focus on the case  $\gamma=0$. In this case we get \eqref{almfinjes} by combining the following facts:
we have the convergence
$Y_\varepsilon(x)
\overset{\varepsilon\rightarrow 0} \longrightarrow Y(x)$ for every $\omega\in \Sigma$
in the $L^\infty$ topology (see Proposition \ref{specialrandom});
we have the following bound
$$\|u_{\varepsilon}(t)\|_{L^2}=\|e^{-Y_{\varepsilon}} v_{\varepsilon}(t)\|_{L^2}
\leq \|e^{-Y_{\varepsilon}}\|_{L^\infty} \| v_{\varepsilon}(t)\|_{L^2}$$
and hence $\|u_{\varepsilon}(t)\|_{L^2}$ is bounded for every $\omega\in \Sigma$ 
by \eqref{prepostego} and Proposition \ref{specialrandom}.

In order to establish \eqref{almfinjes} for $\gamma\in (0,1)$ it is sufficient to interpolate between the convergence for $\gamma=0$ (already established above)  with the uniform bound 
\begin{equation*}
\sup_{\substack{\varepsilon\in (0,1), t\in [-T,T]}}\|e^{Y}(1-e^{Y_{\varepsilon}-Y}) |u_{\varepsilon}(t)|\big \|_{{H^\gamma}}<\infty,
\quad \forall\, \omega\in \Sigma.
\end{equation*}
In order to establish this bound it is sufficient to notice that for every $\omega\in \Sigma$
\begin{equation}\label{5punto6primo}
\sup_{\varepsilon\in (0,1)} \{\|e^{Y}\|_{H^\gamma\cap L^\infty},
\|1-e^{Y_{\varepsilon}-Y} \|_{H^\gamma\cap L^\infty},
\||u_{\varepsilon}(t)|\|_{H^\gamma\cap L^\infty}\}<\infty
\end{equation}
and to recall that $H^\gamma\cap L^\infty$ is an algebra.
We recall that the boundedness of 
$\||u_{\varepsilon}(t)|\|_{H^\gamma\cap L^\infty}$ comes on one hand by combining
\eqref{prepostego} with
\begin{equation*}
\||u_{\varepsilon}(t)|\|_{L^\infty}=\|e^{-Y_{\varepsilon}} |v_{\varepsilon}(t)|\|_{L^\infty}\leq \|e^{-Y_{\varepsilon}}\|_{L^\infty} \|v_{\varepsilon}(t)\|_{L^\infty} 
\lesssim \|v_{\varepsilon}(t)\|_{H^s}, 
\end{equation*}
where $s>1$. On the other hand we have the following computation:
\begin{multline*}
\||u_{\varepsilon}(t)|\|_{H^\gamma}=
\|e^{-Y_{\varepsilon}(x)} |v_{\varepsilon}(t)| \|_{H^\gamma}
\\
\leq \|e^{-Y_{\varepsilon}}\|_{L^\infty\cap H^\gamma} \||v_{\varepsilon}(t)| \|_{L^\infty\cap H^\gamma}
\lesssim \||v_{\varepsilon}(t)| \|_{L^\infty\cap H^\gamma}, 
\quad \gamma\in [0, 1)
\end{multline*} where we have used Proposition \ref{specialrandom}
and hence we get the desired uniform bound since by the diamagnetic inequality
$$\||v_{\varepsilon}(t)| \|_{L^\infty\cap H^\gamma}\leq \||v_{\varepsilon}(t)| \|_{L^\infty\cap H^1}
\lesssim \|v_{\varepsilon}(t)\|_{L^\infty\cap H^1}$$
and we conclude by \eqref{prepostego}.\\

Let us now establish  \eqref{diffnew}. Notice that by combining  \eqref{L2Linf}, \eqref{L2Linfremcat} and \eqref{almfinjes} 
we have:
\begin{equation}\label{diffaltenewstron}
\sup_{t\in [-T,T]} \big \|e^{Y} |u_{\varepsilon}(t)| - | v(t)| \big \|_
{{H^\gamma}\cap L^\infty} \overset{\varepsilon\rightarrow 0}
\longrightarrow 0, \quad \forall \, \omega\in \Sigma,\quad \gamma\in (0,1).
\end{equation}
Hence \eqref{diffnew} in the case $\gamma=0$ and the $L^\infty$ convergence, follow from \eqref{diffaltenewstron} since $e^{-Y}\in L^\infty$ for every $\omega\in \Sigma$
(see Proposition~\ref{specialrandom}). To prove \eqref{diffnew} 
in the general case 
$\gamma\in (0,1)$ it is sufficient to make interpolation between $\gamma=0$ and the bound 
\begin{equation}\label{lastiss}
\sup_\varepsilon \{\|e^{-Y}\|_{H^\gamma\cap L^\infty}, 
\|e^{Y} |u_{\varepsilon}(t)| - | v(t)|
\|_{H^\gamma\cap L^\infty}\}<\infty, \quad \forall \, \omega\in \Sigma \end{equation}
which in turn implies, thanks to the fact that ${H^\gamma\cap L^\infty}$ is an algebra,
that  the quantity $\||u_{\varepsilon}(t)| - e^{-Y}| v(t)|
\|_{H^\gamma}$ is uniformly bounded for every $\omega\in \Sigma$.
The proof of \eqref{lastiss} follows by combining: 
the estimate \eqref{5punto6primo}, the bound $\|v(t)\|_{L^\infty}\lesssim 
\| v(t)\|_{H^s}<C$ for $s\in (1,2)$ where we used \eqref{prepostego} in the last inequality, by the bound \eqref{postego} and finally by the properties of $Y$ 
(see Proposition~ \ref{specialrandom}).
This completes the proof of Theorem~\ref{corproof}.
\section{Proof of Proposition~\ref{modifen}}
In the sequel, we use the following simplified notation:
$v=v_\varepsilon(t,x)$, $Y=Y_\varepsilon(x)$ and ${\mathcal E}={\mathcal E}_{\varepsilon}$. 
Moreover we denote by $(\cdot , \cdot)$ the $L^2$ scalar product. We also drop the explicit dependence
of the functions involved from the variable $(t,x)$, in order to make the computations more compact.
\\

We are interested to construct a suitable energy with the following structure
$$\mathcal{E}(v)=(\Delta v, \Delta v e^{-2Y})+ {\rm remainder}.$$
By using the equation solved by $v$ we have the following identity:
\begin{multline}\label{co}
\frac d{dt} (\Delta v, \Delta v e^{-2Y}) = 2\Re (\partial_t \Delta v,  \Delta v e^{-2Y})
= 2\Re (\partial_t \Delta v, i\partial_t v e^{-2Y} ) 
\\
+ 2\Re (\partial_t \Delta v, 2 \nabla Y \cdot \nabla ve^{-2Y}) 
-2\Re (\partial_t \Delta v, v :|\nabla Y|^2: e^{-2Y})
\\-2\lambda \Re (\partial_t \Delta v,  v|v|^p e^{-(p+2)Y})
=I+II+III+IV.
\end{multline}
Notice that
\begin{multline}\label{I}I=-2\Re (\partial_t \nabla v, i\partial_t \nabla v e^{-2Y} )-2\Re (\partial_t \nabla v, i\partial_t v \nabla (e^{-2Y}) )\\
=-2 \Im  (\partial_t \nabla v, \partial_t v \nabla (e^{-2Y}) ).
\end{multline}
Moreover we have 
\begin{multline}II= 2\Re (\partial_t \Delta v, 2 \nabla Y \cdot \nabla ve^{-2Y})\\
=2\Re (\Delta v, \partial_t \nabla v \cdot \nabla (e^{-2Y})) + 2 \frac d{dt}\Re (\Delta v, 2 \nabla Y \cdot \nabla ve^{-2Y})
\end{multline}
and using again the equation
\begin{multline}II=2 \frac d{dt}\Re (\Delta v, 2 \nabla Y \cdot \nabla ve^{-2Y})+2\Re (i\partial_t v, \partial_t \nabla v \cdot \nabla (e^{-2Y})) 
\\
+4\Re (\nabla Y \cdot \nabla v, \partial_t \nabla v \cdot \nabla (e^{-2Y}))
\\
-
2\Re (v:|\nabla Y|^2: , \partial_t \nabla v \cdot \nabla (e^{-2Y})) 
-2\lambda \Re (e^{-pY}v|v|^p , \partial_t \nabla v \cdot \nabla (e^{-2Y}))
\\
= 2\frac d{dt}\Re (\Delta v, 2 \nabla Y \cdot \nabla ve^{-2Y})-2\Im (\partial_t v, \partial_t \nabla v \cdot \nabla (e^{-2Y})) \\
+4\Re (\nabla Y \cdot \nabla v, \partial_t \nabla v \cdot \nabla (e^{-2Y}))-
2\Re (v :|\nabla Y|^2: , \partial_t \nabla v \cdot \nabla (e^{-2Y})) \\
-2\lambda \Re (e^{-pY} v|v|^p , \partial_t \nabla v \cdot \nabla (e^{-2Y}))
\\
=2 \frac d{dt}\Re (\Delta v, 2 \nabla Y \cdot \nabla ve^{-2Y})+2 \Im (\partial_t \nabla v, \partial_t v \nabla (e^{-2Y})) \\
+4\Re (\nabla Y \cdot \nabla v, \partial_t \nabla v \cdot \nabla (e^{-2Y}))-
2\Re (v :|\nabla Y|^2: , \partial_t \nabla v \cdot \nabla (e^{-2Y})) \\
-2\lambda \Re (e^{-pY}  v|v|^p , \partial_t \nabla v \cdot \nabla (e^{-2Y}))
\end{multline}
and hence by \eqref{I} we get
\begin{multline*}
II  =  2 \frac d{dt}\Re (\Delta v, 2 \nabla Y \cdot \nabla ve^{-2Y})
 -I 
+4\Re (\nabla Y \cdot \nabla v, \partial_t \nabla v \cdot \nabla (e^{-2Y}))
\\
-
2\Re (v :|\nabla Y|^2: , \partial_t \nabla v \cdot \nabla (e^{-2Y})) 
-2\lambda \Re (e^{-pY} v|v|^p , \partial_t \nabla v \cdot \nabla (e^{-2Y}))
\\
= 2\frac d{dt}\Re (\Delta v, 2 \nabla Y \cdot \nabla ve^{-2Y}) 
-I 
+4\Re (\nabla Y \cdot \nabla v, \partial_t \nabla v \cdot \nabla (e^{-2Y}))
\\
-2 \frac d{dt} \Re (v :|\nabla Y|^2: , \nabla v \cdot \nabla (e^{-2Y})) 
+2\Re (\partial_tv:|\nabla Y|^2:, \nabla v \cdot \nabla (e^{-2Y})) 
\\
-2\lambda \Re (e^{-pY}v|v|^p , \partial_t \nabla v \cdot \nabla (e^{-2Y})).
\end{multline*}
Summarizing we get from the previous chain of identities
\begin{multline}\label{cor1}
I+II = 2\Re (\partial_tv:|\nabla Y|^2:, \nabla v \cdot \nabla (e^{-2Y}))
\\
+ 2 \frac d{dt}\Re (\Delta v, 2 \nabla Y \cdot \nabla ve^{-2Y}) 
+4\Re (\nabla Y \cdot \nabla v, \partial_t \nabla v \cdot \nabla (e^{-2Y}))
\\
-
2\frac d{dt} \Re (v :|\nabla Y|^2: , \nabla v \cdot \nabla (e^{-2Y})) 
-2\lambda \Re (e^{-pY} v|v|^p , \partial_t \nabla v \cdot \nabla (e^{-2Y})).
\end{multline}
On the other hand we can compute
 \begin{multline}\label{cor}
 -2\lambda \Re (e^{-pY} v|v|^p , \partial_t \nabla v \cdot \nabla (e^{-2Y}))\\
 = 2\lambda \Re (\nabla ( e^{-pY} v|v|^p) , \partial_t \nabla v e^{-2Y})
 +2\lambda \Re 
(e^{-pY} v|v|^p , \partial_t \Delta v e^{-2Y})
\\
= 2\lambda \Re (\nabla ( e^{-pY}v|v|^p) , \partial_t \nabla v e^{-2Y})
-IV
\\
= 2\lambda \Re (e^{-pY} \nabla v|v|^p, \partial_t \nabla v e^{-2Y})
+2\lambda \Re (e^{-pY}v\nabla (|v|^p), \partial_t \nabla v e^{-2Y})\\
+2\lambda \Re (\nabla(e^{-pY})
v |v|^p, \partial_t \nabla v e^{-2Y})
-IV
\end{multline}
and hence
\begin{multline}\label{vir}
\dots =\lambda \Re (\partial_t (|\nabla v|^2)|v|^p,  e^{-(p+2)Y})
\\+2\lambda \Re (v\nabla (|v|^p), \partial_t \nabla v  e^{-(p+2)Y})+2\lambda \Re (\nabla(e^{-pY})v |v|^p, \partial_t \nabla v e^{-2Y})
-IV
\\
=\lambda \Re (\partial_t (|\nabla v|^2)|v|^p,  e^{-(p+2)Y})
+2\lambda \frac d{dt} \Re (v\nabla (|v|^p), \nabla v  e^{-(p+2)Y})
-2\lambda \Re (\partial_t v\nabla (|v|^p), \nabla v  e^{-(p+2)Y})\\
-2\lambda \Re (v\nabla \partial_t (|v|^p), \nabla v  e^{-(p+2)Y})
+2\lambda \Re (\nabla(e^{-pY})v |v|^p, \partial_t \nabla v e^{-2Y})
-IV
\\
=\lambda \frac d{dt} \Re (|\nabla v|^2|v|^p,  e^{-(p+2)Y})
-\lambda \Re (|\nabla v|^2 \partial_t (|v|^p),  e^{-(p+2)Y})
+2\lambda \frac d{dt} \Re (v\nabla (|v|^p), \nabla v  e^{-(p+2)Y})\\
-2\lambda \Re (\partial_t v\nabla (|v|^p), \nabla v  e^{-(p+2)Y})
-\frac{\lambda p}{2} (\partial_t (\nabla(|v|^2)|v|^{p-2}), \nabla (|v|^2)  e^{-(p+2)Y})
\\+2\lambda \Re (\nabla(e^{-pY})v |v|^p, \partial_t \nabla v e^{-2Y})
-IV
\\
=\lambda \frac d{dt}  \Re (|\nabla v|^2 |v|^p,  e^{-(p+2)Y})
-\lambda \Re (|\nabla v|^2 \partial_t (|v|^p),  e^{-(p+2)Y})
+2\lambda \frac d{dt} \Re (v\nabla (|v|^p), \nabla v  e^{-(p+2)Y})\\
-2\lambda \Re (\partial_t v\nabla (|v|^p), \nabla v  e^{-(p+2)Y})
-\frac {\lambda p}{2}(\partial_t \nabla(|v|^{2})|v|^{p-2}, \nabla (|v|^2)  e^{-(p+2)Y})
\\-\frac{\lambda p}{2}  (\nabla(|v|^2) \partial_t (|v|^{p-2}), \nabla (|v|^2)  e^{-(p+2)Y})
+2\lambda \Re (\nabla(e^{-pY})v |v|^p, \partial_t \nabla v e^{-2Y})
-IV\\
=\lambda \frac d{dt}  \Re (|\nabla v|^2|v|^p,  e^{-(p+2)Y})
-\lambda \Re (|\nabla v|^2 \partial_t (|v|^p),  e^{-(p+2)Y})
+2\lambda \frac d{dt} \Re (v\nabla (|v|^p), \nabla v  e^{-(p+2)Y})\\
-2\lambda \Re (\partial_t v\nabla (|v|^p), \nabla v  e^{-(p+2)Y})
-\frac {\lambda p}4 (\partial_t (|\nabla(|v|^2)|^2)|v|^{p-2},  e^{-(p+2)Y})
\\- \frac{\lambda p}2 (\nabla(|v|^2) \partial_t (|v|^{p-2}), \nabla (|v|^2)  e^{-(p+2)Y})
+2\lambda \Re (\nabla(e^{-pY})v |v|^p, \partial_t \nabla v e^{-2Y})
-IV
\\
=\lambda \frac d{dt}  \Re (|\nabla v|^2|v|^p,  e^{-(p+2)Y})
-\lambda \Re (|\nabla v|^2 \partial_t (|v|^p), e^{-(p+2)Y}) 
+2\lambda \frac d{dt} \Re (v\nabla (|v|^p), \nabla v  e^{-(p+2)Y})\\
-2\lambda \Re (\partial_t v\nabla (|v|^p), \nabla v  e^{-(p+2)Y})
-\frac {\lambda p} 4 \frac d{dt}(|\nabla(|v|^2)|^2|v|^{p-2},  e^{-(p+2)Y})\\
+\frac {\lambda p}4 (|\nabla(|v|^2)|^2\partial_t (|v|^{p-2}),  e^{-(p+2)Y}) 
-\frac{\lambda p}2(\nabla(|v|^2) \partial_t (|v|^{p-2}), \nabla (|v|^2)  e^{-(p+2)Y})\\
+2\lambda \Re (\nabla(e^{-pY})v |v|^p, \partial_t \nabla v e^{-2Y})-IV\,. 
\end{multline}
By combining \eqref{cor1}, \eqref{cor} and \eqref{vir} we get
\begin{multline}\label{cop} I+II+IV
=2\Re (\partial_tv:|\nabla Y|^2:, \nabla v \cdot \nabla (e^{-2Y}))
\\+ 2 \frac d{dt}\Re (\Delta v, 2 \nabla Y \cdot \nabla ve^{-2Y}) 
+4\Re (\nabla Y \cdot \nabla v, \partial_t \nabla v \cdot \nabla (e^{-2Y}))
\\-2\frac d{dt} \Re (v :|\nabla Y|^2: , \nabla v \cdot \nabla (e^{-2Y}))\\
+\lambda \frac d{dt}  \Re (|\nabla v|^2|v|^p, e^{-(p+2)Y})
-\lambda \Re (|\nabla v|^2 \partial_t (|v|^p), e^{-(p+2)Y}) 
+2\lambda \frac d{dt} \Re (v\nabla (|v|^p), \nabla v  e^{-(p+2)Y})\\
-2\lambda \Re (\partial_t v\nabla (|v|^p), \nabla v  e^{-(p+2)Y})
-\frac {\lambda p} 4 \frac d{dt}(|\nabla(|v|^2)|^2|v|^{p-2},  e^{-(p+2)Y})\\
+\frac {\lambda p}4 (|\nabla(|v|^2)|^2\partial_t (|v|^{p-2}),  e^{-(p+2)Y}) 
-\frac{\lambda p}2(\nabla(|v|^2) \partial_t (|v|^{p-2}), \nabla (|v|^2)  e^{-(p+2)Y})\\
+2\lambda \Re (\nabla(e^{-pY})v |v|^p, \partial_t \nabla v e^{-2Y}).
\end{multline}
Next notice that
\begin{equation}\label{viru}III= -2\frac d{dt}\Re (\Delta v, v :|\nabla Y|^2: e^{-2Y})+ 
2 \Re (\Delta v, \partial_t v :|\nabla Y|^2: e^{-2Y}).
\end{equation}
Summarzing we get
\begin{multline}\label{cagsot}I+II+III+IV
=2\Re (\partial_tv:|\nabla Y|^2:, \nabla v \cdot \nabla (e^{-2Y}))
\\+ 2 \frac d{dt}\Re (\Delta v, 2 \nabla Y \cdot \nabla ve^{-2Y}) +
4\Re (\nabla Y \cdot \nabla v, \partial_t \nabla v \cdot \nabla (e^{-2Y}))
\\-2\frac d{dt} \Re (v :|\nabla Y|^2: , \nabla v \cdot \nabla (e^{-2Y}))
+\lambda \frac d{dt}  \Re (|\nabla v|^2|v|^p,  e^{-(p+2)Y})\\
-\lambda \Re (|\nabla v|^2 \partial_t (|v|^p),  e^{-(p+2)Y}) 
+2\lambda \frac d{dt} \Re (v\nabla (|v|^p), \nabla v  e^{-(p+2)Y})\\
-2\lambda \Re (\partial_t v\nabla (|v|^p), \nabla v  e^{-(p+2)Y})
-\frac {\lambda p} 4 \frac d{dt}(|\nabla(|v|^2)|^2|v|^{p-2},  e^{-(p+2)Y})\\
+\frac {\lambda p}4 (|\nabla(|v|^2)|^2\partial_t (|v|^{p-2}),  e^{-(p+2)Y}) 
-\frac{\lambda p}2(\nabla(|v|^2) \partial_t (|v|^{p-2}), \nabla (|v|^2)  e^{-(p+2)Y})\\
+2\lambda \frac d{dt} \Re (\nabla(e^{-pY})v |v|^p, \nabla v e^{-2Y})
-2\lambda \Re (\nabla(e^{-pY})\partial_t (v |v|^p), \nabla v e^{-2Y})\\
-2\frac d{dt}\Re (\Delta v, v :|\nabla Y|^2: e^{-2Y})+ 
2 \Re (\Delta v, \partial_t v :|\nabla Y|^2: e^{-2Y}).
\end{multline}
Next by using the equation we compute the first and last term on the r.h.s. in \eqref{cagsot} as follows:
\begin{multline}\label{cagsot2}
2\Re (\partial_tv:|\nabla Y|^2:, \nabla v \cdot \nabla (e^{-2Y}))
+ 
2 \Re (\Delta v, \partial_t v :|\nabla Y|^2: e^{-2Y})
\\= 2\Re (\partial_tv:|\nabla Y|^2:, \nabla v \cdot \nabla (e^{-2Y}))+ 2 \Re (2\nabla v \cdot \nabla Y, 
\partial_t v :|\nabla Y|^2: e^{-2Y})
\\- 2 \Re (v:|\nabla Y|^2:, \partial_t v :|\nabla Y|^2: e^{-2Y})- 2 \lambda \Re ( e^{-pY}v|v|^p, \partial_t v :|\nabla Y|^2:  e^{-2Y})
 \\
 = -4\Re (\partial_t v :|\nabla Y|^2:, \nabla v \cdot \nabla Y e^{-2Y})+ 4 \Re (\nabla v\cdot  \nabla Y, 
\partial_t v :|\nabla Y|^2: e^{-2Y})
\\- 2 \Re (v:|\nabla Y|^2:, \partial_t v :|\nabla Y|^2: e^{-2Y})- 2 \lambda \Re (v|v|^p, \partial_t v :|\nabla Y|^2:   e^{-(p+2)Y})
\\=-\frac d{dt} (|v|^2:|\nabla Y|^2:, :|\nabla Y|^2: e^{-2Y})
- \frac {2 \lambda}{p+2} \frac d{dt}\Re (|v|^{p+2}, :|\nabla Y|^2:  e^{-(p+2)Y}).
\end{multline}
Finally we show that the third term on the r.h.s. in \eqref{cagsot} can be written as a total derivative w.r.t. time variable:
\begin{multline}\label{cagsot1}4\Re (\nabla Y \cdot \nabla v, \partial_t \nabla v \cdot \nabla (e^{-2Y}))\\
=-8 \Re \sum_{i,j=1}^2 \int_{\T^2} e^{-2Y} \partial_i Y \partial_i v \partial_t \partial_j \bar v \partial_j Y
=-4 \Re \sum_{i=1}^2 \int_{\T^2} e^{-2Y}  \partial_t (|\partial_i v|^2) (\partial_i Y)^2
\\-8 \Re \int_{\T^2} e^{-2Y} \partial_1 Y \partial_1 v \partial_t \partial_2 \bar v \partial_2 Y
-8 \Re \int_{\T^2} e^{-2Y} \partial_2 Y \partial_2 v \partial_t \partial_1 \bar v \partial_1 Y 
\\
=-4 \frac d{dt}  \sum_{i=1}^2 \int_{\T^2} e^{-2Y}  (|\partial_i v|^2) (\partial_i Y)^2 
-8 \Re \int_{\T^2} e^{-2Y} \partial_1 Y \partial_2 Y (\partial_1 v \partial_t \partial_2 \bar v 
+\partial_2 v \partial_t \partial_1 \bar v)
\\
=-4 \frac d{dt}  \sum_{i=1}^2 \int_{\T^2} e^{-2Y}  (|\partial_i v|^2) (\partial_i Y)^2 
-8 \Re \int_{\T^2} e^{-2Y} \partial_1 Y \partial_2 Y \partial_t (\partial_1 v \partial_2 \bar v) 
\\
=-4 \frac d{dt}  \sum_{i=1}^2 \int_{\T^2} e^{-2Y}  (|\partial_i v|^2) (\partial_i Y)^2 
-8 \frac d{dt} \Re \int_{\T^2} e^{-2Y} \partial_1 Y \partial_2 Y \partial_1 v \partial_2 \bar v.
\end{multline}
We conclude the proof of Proposition~\ref{modifen} by combining \eqref{co}, \eqref{cagsot}, \eqref{cagsot2}, \eqref{cagsot1}.

\end{document}